\documentclass[12pt,a4paper]{amsart}
\usepackage{mathrsfs}
\usepackage{amssymb,amsmath,amsthm,color, bm}
\usepackage{graphicx,mcite, cite}
\usepackage{hyperref}
\usepackage{url}
\usepackage{setspace}
\usepackage{enumerate}
\usepackage{cite}
\usepackage{tikz}
\usepackage{graphicx} 

\newtheorem{theorem}{Theorem}[section]
\newtheorem{lemma}[theorem]{Lemma}
\newtheorem{proof of lemma}[theorem]{Proof of Lemma}
\newtheorem{proposition}[theorem]{Proposition}

\newtheorem{corollary}[theorem]{Corollary}

\theoremstyle{definition}
\newtheorem{definition}[theorem]{Definition}

\makeindex

\newtheorem{remark}[theorem]{Remark}

\numberwithin{equation}{section}

\begin{document}
\title[Pointwise convergence for the heat equation on $\mathbb T^n$ and $\mathbb T^n \times \mathbb R^m$]{Pointwise convergence for the heat equation on tori  $\mathbb T^n$ and waveguide manifold $\mathbb T^n \times \mathbb R^m$}

\author{Divyang G. Bhimani and Rupak K. Dalai}

\address{Department of Mathematics, Indian Institute of Science Education and Research-Pune, Homi Bhabha Road, Pune 411008, India}

\email{divyang.bhimani@iiserpune.ac.in, rupakinmath@gmail.com}

\subjclass[2010]{Primary 42B25, 35K08, 40A10 ; Secondary 35C15, 35G10}

\date{}

\keywords{Heat kernel, Pointwise convergence, maximal operator, weighted inequalities, waveguide manifold.}

\maketitle

\begin{abstract} 
 We completely characterize the  weighted Lebesgue spaces on the
torus $\mathbb T^n$ and waveguide manifold $\mathbb T^n \times \mathbb R^m$ for which the solutions of the heat equation converge pointwise (as time tends to zero) to the initial data.   In the process, we also characterize the weighted Lebesgue spaces for the boundedness of maximal operators on the torus and waveguide manifold, which may be of independent interest.
\end{abstract}

\tableofcontents

\section{Introduction}
\subsection{Heat equation on tori $\mathbb T^n$}We consider the  heat equation on the torus
\begin{equation}\label{ht}
\begin{cases} 
{ \frac { \partial u } { \partial t } ( x , t ) = \Delta _ { x } u ( x , t ) } \\
{ u ( x , 0 ) = f ( x ) }
\end{cases} (x, t)\in \mathbb T^n \times\mathbb R^+.
\end{equation}
Formally,   the  solution of \eqref{ht}   is given by
\[u(x, t)= \mathcal{F}_T^{-1}\left(e^{-t|l|^2} \mathcal{F}_Tf(l)\right)(x)=\phi_t\ast f(x) \quad  (l \in \mathbb Z^n,  x \in \mathbb T^n),\]
where $\mathcal{F}_{T}$ and $\mathcal{F}^{-1}_T$ are the toroidal Fourier  and inverse Fourier transforms respectively (see Section \ref{pre})   and $\ast$ denotes convolution on $\mathbb T^n.$  In view of  Poisson summation formula,
we may rewrite the heat kernel  as follows
\begin{equation}\label{ehk}
\phi_t(x)= \mathcal{F}_T^{-1} (e^{-t|l|^2})(x)= \sum_{l \in \mathbb{Z}^n} e^{i l \cdot x - |l|^2 t}= (4\pi t)^{-\frac{n}{2}} \sum_{k \in \mathbb{Z}^n} e^{-\frac{|x+k|^2}{4 t}}. 
\end{equation}
Given initial data  $f\in L^p(\mathbb T^n) \  (1\leq p \leq \infty),$ it is known that the solution   $u(x,t)$ of \eqref{ht} converges  pointwise  to $f(x)$  for almost all $x$ as $t$ tends to $0$ (see e.g. Theorem $8.15$ in \cite{raFolland}). Specifically,  for $f\in L^p(\mathbb T^n),$ the following limit holds
\begin{equation}\label{limpt}
\quad \lim _{t \rightarrow 0^+} u(x, t)=f(x) \mbox{ for almost every} \  x.
\end{equation} Let $v$  be a  strictly  positive weight on the torus $\mathbb T^n.$  Define weighted Lebesgue space for $1\leq p<\infty$ by 
\[L^p_v(\mathbb T^n)= \left\{ f: \mathbb T^n \to \mathbb C\left|\|f\|_{L^p_v(\mathbb T^n)}= \|f v^{\frac{1}{p}} \|_{L^p(\mathbb T^n)}< \infty\right. \right\}. \]
\begin{definition}\label{wct}  Let $v$ be a strictly positive weight on the torus $\mathbb{T}^n$,  $1\leq p<\infty$ and $\frac{1}{p}+\frac{1}{p'}=1.$
	We define that the weight $v$ belongs to the class $D^{T}_p$  if there exists $t_0>0$ such that
	\begin{equation}\label{eq16}
		D_p^{T}:=\left\{v: v^{-\frac{1}{p}}\phi_{t_0} \in L^{p^{\prime}}(\mathbb T^n)\right\}.
	\end{equation}
\end{definition}
It is now natural to ask if we can characterize weighted  Lebesgue  spaces $L^p_v(\mathbb T^n)$ so that the limit \eqref{limpt}
holds?  We answer this question by establishing the  following theorem.
\begin{theorem}\label{th1}
Let $v$  be a  strictly  positive weight on the torus $\mathbb T^n$ and  $f\in L^p_v(\mathbb T^n).$ Suppose $0<T< \frac{1}{2}.$ Then $v\in D^{T}_p$ if and only if $u(x, t)$ is a solution of (\ref{ht}) for $t\in (0,T],$ and \[\lim _{t \rightarrow 0^+} u(x,t)=f(x) \text{ for almost every }  x. \] 
\end{theorem}

In recent years, there has been a great deal of interest to study the solutions to a given Cauchy problem converge pointwise almost everywhere to the initial data,  see \cite{vivianiPAMS, stingaPA, torreaTAMS, CardosoARXIV, BetancorARXIV, BrunoARXIV, Flores2023}. Hartzstein-Torrea-Viviani \cite{vivianiPAMS} have  characterized the weighted Lebesgue spaces  on $\mathbb R^m$   for which the solutions of the heat and Poisson  equations have limits a.e.  as  $t\to 0$.   In \cite{stingaPA}, Abu-Falahah-Stinga-Torrea  have addressed similar question for   heat-diffusion problems associated to the harmonic oscillator and the Ornstein-Uhlenbeck operator.  Later  in \cite{torreaTAMS},   Garrig\'os, Hartzstein, Signes, Torrea and Viviani  have studied  similar problem for heat and Poisson in the presence of harmonic potentials.  Recently  Cardoso \cite{CardosoARXIV}, Bruno-Papageorgiou \cite{BrunoARXIV}, and Romero-Barrios-Betancor \cite{BetancorARXIV} have studied similar problem  on the Heisenberg group, symmetric spaces, and homogeneous trees, respectively. However,  these works did not consider the torus case, 
and thus Theorem \ref{th1} is new. 

\subsection{Heat equation on waveguide manifold $\mathbb T^n \times \mathbb R^m$}
We consider the heat equation on the waveguide manifold
\begin{equation}\label{eq-r11}
\begin{cases} \frac{\partial u}{\partial t}(x, y, t)=\Delta_{WG} \,u(x, y, t)\\
u(x, y, 0)=f(x, y)
\end{cases}(x, y, t) \in \mathbb{T}^n \times\mathbb{R}^m\times\mathbb R^+.
\end{equation}
Here,  $\Delta_{W G}=\Delta_{\mathbb T^{n}}+\Delta_{\mathbb{R}^{m}}$, where $\Delta_{\mathbb T^{n}}$ represents the Laplace-Beltrami operator on $\mathbb T^{n}$ and $\Delta_{\mathbb{R}^{m}}$  Laplacian  on  $\mathbb{R}^{m}$.  
 The product space  $\mathbb T^n \times \mathbb R^m$ is known as semi-periodic space or waveguide manifold.  $d = n+m$  is the whole dimension while $n$ is the dimension for the tori component and $m$ is the dimension for the Euclidean component.   Semiperiodicity, which is a partial confinement that forces the wave function to periodically move along certain directions, is one of the most intriguing features of model \eqref{eq-r11}. This particularly leads to some unexpected new challenges for mathematical analysis.  In today’s backbone networks, data signals are almost exclusively transmitted by optical carriers in fibers. Applications like the internet demand an increase of the available bandwidth in the network and reduction of costs for the transmission of data.  Waveguide manifolds  are of particular interest in nonlinear optics of telecommunications,  see  \cite{schneider2013nonlinear, snyder1983optical}. 

\smallskip

Formally,  the  solution of  (\ref{eq-r11}) is expressed as follows
$$
\begin{aligned}
	u(x, y, t) & =e^{-t \Delta_{W G}} f(x, y) \\
	& =\sum_{k \in \mathbb Z^{n}} \int_{\mathbb{R}^{m}}e^{i(k\cdot x+\xi \cdot y)} e^{-t\left(|k|^{2}+|\xi|^{2}\right)} \mathcal{F}_{W G} f(k,\xi)\,d\xi \\
	& =\mathcal{F}_{W G}^{-1}\left(e^{-t\left(|k|^{2}+|\xi|^{2}\right)} \mathcal{F}_{W G} f(k,\xi)\right)(x, y) \\
	& =\phi_{t}^{W G} * f(x, y),
\end{aligned}
$$
where $\ast$ denote  convolution on $\mathbb T^n \times \mathbb R^m$ and $\mathcal{F}_{WG}$ and $\mathcal{F}^{-1}_{WG}$ are the waveguide Fourier  and inverse Fourier transforms respectively (see Section \ref{pre}).
\begin{definition}\label{def-r1}  Let $v$ be a strictly positive weight on the waveguide $\mathbb{T}^n \times\mathbb{R}^m$.
We define that the weight $v$ belongs to the class $D^{W G}_p$ for $1\leq p<\infty$ if there exists $t_0>0$ such that
$$
\int_{\mathbb{T}^n \times\mathbb{R}^m}\left|v^{-\frac{1}{p}}(x, y) e^{-\frac{|y|^{2}}{4 t_{0}}} \right|^{p'}d x \,d y<\infty .
$$
\end{definition}

\begin{theorem}\label{th-r1}
	Let $v$ be a strictly positive weight on the waveguide $\mathbb{T}^n \times\mathbb{R}^m$ and $f\in L^p_{v}\left(\mathbb{T}^n \times\mathbb{R}^m\right).$ Suppose $0<T<\frac{1}{2}.$ Then $v\in D^{W G}_p$ if and only if $u(x, y, t)$ is a solution of (\ref{eq-r11}) for $t\in(0,T]$, and
	$$
	\lim _{t \rightarrow 0^+} u(x, y, t)=f(x, y) \text { for almost every } (x, y).
	$$
\end{theorem}

In recent years, there has been a great deal of interest to study dispersive PDEs on waveguide manifold,   see e.g. \cite{Luo, Yu,  Hani} and the references therein.  On the other  hand,  heat equation on waveguide manifold is understudied and  to the best of authors knowledge,   Theorem \ref{th-r1} is the first result for  heat equation on waveguide manifold.

\subsection{Comments on Theorems \ref{th1} and \ref{th-r1}}
Our strategy of  proof is inspired from the seminal work of Hartzstein et al.  in \cite{vivianiPAMS}, where they addressed similar  problem on Euclidean space $\mathbb R^m.$ However,   we  would like to mention that heat kernel/maximal operator  on torus/waveguide manifold is significantly different than on Euclidean space,   and we  shall need some novel tool to treat the problem.   We note that  pointwise convergence  for \eqref{limpt} and boundedness of maximal operator are closely related.  More specifically,  in order to prove Theorems \ref{th1} and \ref{th-r1},   we prove Theorems \ref{th2} and \ref{th-r2}.   See also Remarks \ref{fr1} and \ref{fr2}.
 Let us now briefly mention the   circle of ideas   to  characterized the weighted Lebesgue spaces in which   a.e.  convergence holds (i.e.  strategy to prove Theorems \ref{th2} and \ref{th-r2}):
\begin{itemize}
	\item[-]Define local maximal operator by
	\[
	H^*_Rf(x) = \sup_{0<t<R}\left|\phi_t\ast f(x)\right| \   \quad (0<R<\frac{1}{2})
	\]
and consider the maximal operator $\mathcal{M}^{T}$ on torus (see Definition \ref{HMd}).
\item[-]Since \( H_R^* \) can only be dominated by  \( \mathcal{M}^{T} \) if the heat kernel \( \phi_t \) is supported within a suitable ball, we decompose \( \phi_t \) into two parts.  First part is supported within the required ball, allowing us to invoke a crucial upper bound on \( \phi_t \), thereby \( H_R^* \) is dominated by \( \mathcal{M}^{T} \) (see \eqref{eq-r2}). The second part is dominated by \( \phi_R * f \) due to the increasing nature of \( \phi_t \) over the suitable time interval (see \ref{eq12}).	

\item[-] This essentially reduce our analysis to  characterize boundedness of $\mathcal{M}^{T}$.   See Proposition \ref{GP} and Remark \ref{PSGP} for its proof strategy.  Using this together with Proposition  \ref{prop1},  we conclude that   \( H_R^* \) maps \( L_v^p(\mathbb{T}^n) \) into \( L_u^p(\mathbb{T}^n) \) for some weight \( u \).	
\item[-]In order to have pointwise convergence it is sufficient  to have weak boundedness of \( H_R^* \).
\item[-]Conversely, pointwise convergence implies that \( H_R^*f \) is finite almost everywhere, ensuring that \( f \) belongs to the required weighted space \( L_v^p(\mathbb{T}^n) \) (see the proof of Proposition \ref{prop1}).	
\item[-] Overall  proof strategy for waveguide manifold is 	essentially the same as torus case.  However,  we note that there are several technical differences which one  needs to handle carefully.
\end{itemize}

We conclude this section by adding several further remarks:

\begin{remark}\label{rk1}
By duality argument, we have the following equalities:
$$ 
\begin{aligned}&\bigcup_{\substack{v\in D^{T}_p, \\  p\in[1,\infty)}}L^p_v(\mathbb T^n)= \left\{ f: \mathbb T^n \to \mathbb  C \left| \int_{\mathbb{T}^n} |f(x)| \phi_{t_0}(x) \, dx < \infty \right.\right\}= L^1(\mathbb T^n),\ and\ \\
 &\bigcup_{\substack{v\in D^{WG}_p, \\  p\in[1,\infty)}}L^p_v(\mathbb{T}^n \times \mathbb{R}^m)= \left\{ f: \mathbb{T}^n \times \mathbb{R}^m \to \mathbb  C \left| \int_{\mathbb{T}^n \times \mathbb{R}^m} |f(x,y)|e^{\frac{-|y|^2}{4 t_0}} dxdy < \infty \right.\right\}.
 \end{aligned}$$
\end{remark}

\begin{remark} In \cite{BhimaniNoDEA},  the first author have established finite time-blow up for the fractional  heat equation on torus.   In fact,  many authors have studied Cauchy problem for heat equation associated to  fractional differential operators,  see e.g.  \cite{Flores2023, BhimaniFCAA, BhimaniAdv}.   The analogue of Theorems \ref{th1} and \ref{th-r1} remains open for fractional Laplacians.
\end{remark}    
\begin{remark}  In \cite{torreaTAMS, Flores2023},  authors have studied pointwise convergence   for  heat equation in the presence of potential on $\mathbb R^m.$ Bourgain \cite{Bourgain} have studied  nonlinear Schr\"odinger equations  in the presence of external potentials.   On the other hand,  the pointwise convergence problem   on torus in the presence of potential    remains interesting  open question.  
\end{remark}




\begin{remark}   Although $\mathbb T^n$ is commonly thought of a compact set in $\mathbb R^n$,  the behavior of balls on $\mathbb T^n$ differs from that of balls in the corresponding compact set.  Thus,  the  maximal operator on $\mathbb T^n$ differs from the local maximal operator $\mathcal{M}^T$ on  $\mathbb R^n$ and required careful analysis to establish boundedness of $\mathcal{M}^T$ (see Proposition \ref{GP}).  Moreover, in the context of waveguide  $\mathbb T^n \times \mathbb R^m$,  the solution to the heat equation exhibits a mixed nature,  where each variable behaves according to its respective space.  In this scenario, both cases need to be addressed simultaneously.
\end{remark}

The rest of the paper is organized as follows. In Section \ref{pre}, we describe our setting and provide some estimates of the heat kernels on both \(\mathbb{T}^n\) and \(\mathbb{T}^n \times \mathbb{R}^m\). We also recall some results that will be essential for proving our main results. In Section \ref{mfe}, we characterize the weighted Lebesgue spaces for which we establish the boundedness of maximal operators on both spaces. In Section \ref{mrs}, we present some necessary results and prove our main Theorems \ref{th1} and \ref{th-r1}.

\section{Preliminaries}\label{pre}
The notation $A \lesssim B $ means $A \leq cB$ for a some constant $c > 0 $, whereas $ A \asymp B $ means $c^{-1}A\leq B\leq cA $ for some $c\geq 1$.  The characteristic function  $\chi_{E}$ of a set  $E\subset \mathbb R^n$    is $\chi_{E}(x)=1$ if $x\in E$ and $\chi_{E}(x)=0$ if $x\notin E.$
We are starting by noting that there is a one-to-one correspondence  between functions on $\mathbb R^{n}$ that are 1-periodic in each of the coordinate and functions on torus $\mathbb T^n,$ and we may identify $\mathbb T^{n}=\mathbb R^{n}/\mathbb Z^{n}$ with unit cube $Q_n:=[-1/2 , 1/2)^{n}.$ Let $\mathcal{D}(\mathbb T^n)$ be the vector space  $C^{\infty}(\mathbb T^{n})$ endowed with the usual test function topology, and let $\mathcal{D}'(\mathbb T^{n})$ be its dual, the space of distributions on $\mathbb T^{n}$. Let $\mathcal{S}(\mathbb Z^{n})$ denote the space of rapidly decaying functions $\mathbb Z^{n} \to \mathbb C.$ Let $\mathcal{F}_T:\mathcal{D}(\mathbb T^{n}) \to \mathcal{S}(\mathbb Z^{n})$  be the toroidal Fourier transform (hence the subscript $T$) defined by
$$(\mathcal{F}_Tf)(k):=\hat{f}(k)= \int_{\mathbb T^n}f(x)e^{-i k \cdot x} dx , \  \ (k \in \mathbb Z^{n}).$$
Then $\mathcal{F}_T$ is a bijection and  the inverse Fourier transform is given by
$$ (\mathcal{F}^{-1}_Tf)(x):= \sum_{k \in \mathbb Z^{n}} \hat{f}(k) e^{i k\cdot x}, \   \ (x\in \mathbb T^{n})$$
and this Fourier transform is extended uniquely to $\mathcal{F}_T:\mathcal{D}'(\mathbb T^n) \to \mathcal{S}'(\mathbb Z^n)$.   See \cite{Rbook} for more details.
Let $\mathcal{F}_{WG}:\mathcal{D}(\mathbb T^{n} \times \mathbb R^m ) \to \mathcal{S}(\mathbb Z^{n} \times \mathbb R^m)$  be the waveguide Fourier transform (hence the subscript $WG$) defined by
\[(\mathcal{F}_{WG}f) (k, \xi):= \hat{f}(k, \xi)= \int_{\mathbb T^n \times \mathbb R^m}f(x, y)e^{-i (k \cdot x + \xi \cdot y)} dx\, dy , \  \ (k,  \xi) \in \mathbb Z^{n}\times \mathbb R^m. \]
Then $\mathcal{F}_{WG}$ is a bijection and  the waveguide  inverse  Fourier transform is given by
$$ (\mathcal{F}^{-1}_{WG}f)(x, y):= \sum_{k \in \mathbb Z^{n}} \int_{\mathbb R^n} \hat{f}(k, \xi) e^{i ( k\cdot x + \xi \cdot y)} d\xi,  \    (x, y) \in \mathbb T^{n} \times \mathbb R^m.$$
For $1\leq p<\infty,$ define  the space $L^p_{loc}(\mathbb T^n \times \mathbb R^m)=\{ f: \mathbb T^n \times \mathbb R^m \to \mathbb C \ \text{measurable}: \int_K |f|^p \, d\mu < \infty  \ \forall K \subset \mathbb T^n \times \mathbb R^m,   K \ \text{compact} \}.$ Let \( X \) denote either \( \mathbb{T}^n \) or \( \mathbb{T}^n \times \mathbb{R}^m \). A measurable function \( f \) is said to be in the  weak \( L^p(X) \) space if there exists a constant \( C > 0 \) such that, for all \( \lambda > 0 \),
	\[
	\mu\{x \in X : |f(x)| > \lambda\} \leq \frac{C^p}{\lambda^p},
	\]
	where \( \mu \) is the measure on \( X \).

We recall following  abstract  lemmas, which will be useful to prove  Proposition \ref{GP}.

\begin{lemma}[see Theorem 1.1 in  \cite{MR1240931}] \label{lm-r1}
Let $(Y, \mu)$ be a measurable space, $\mathcal{B}_1$ and $\mathcal{B}_2$ be Banach spaces. Suppose $T$  be a sublinear operator from $T: \mathcal{B}_1 \rightarrow \mathcal{B}_2$ for some $0<s<p<\infty$, satisfying
$$
\left\|\left(\sum_{j \in \mathbb{Z}}\left\|T f_{j}\right\|_{\mathcal{B}_2}^{p}\right)^{\frac{1}{p}}\right\|_{L^{s}(Y)} \leq C_{p, s}\left(\sum_{j \in \mathbb{Z}}\left\|f_{j}\right\|_{\mathcal{B}_1}^{p}\right)^{\frac{1}{p}},
$$
where $C_{p, s}$ is a constant depending on $Y, \mathcal{B}_1, \mathcal{B}_2,p$ and $s$.
Then there exists a positive function $u$ such that $u^{-1} \in L^{\frac{s}{p}}(Y),\left\|u^{-1}\right\|_{L^{\frac{s}{p}}(Y)} \leq 1$ and
$$
\int_{Y}|T f(x)|^{p} u(x) d \mu(x) \leq C\|f\|_{\mathcal{B}_1}^{p}
$$
for some constant $C$ depending on $Y, \mathcal{B}_1, \mathcal{B}_2,p$ and $s$.
\end{lemma}
\begin{lemma}[Kolmogorov inequality,  see Theorem 3.3.1,  p. 59 in \cite{MR0596037}]\label{ki}
Let $\Omega\subseteq\mathbb{R}^{n}$ and $T$ be a sublinear operator maps \(L_v^p(\Omega)\) into weak \(L_u^p(\Omega)\) for some weights \(u, v\) and \(1 \leq p < \infty\). Then given $ 0<s<p$, there exists a constant $C$ such that for every subset $A \subset \Omega$ with $u(A):=\int_\Omega \chi_A(x)u(x)dx<\infty$, we have
\begin{equation*}
\|Tf\|_{L^s_u{(A})} \leq C \,u(A)^{\frac{1}{s}-\frac{1}{p}}\,\|f\|_{L_v^p(\Omega)}.
\end{equation*}
\end{lemma}
\subsection{Heat kernel estimates}
In the following theorem, we establish the crucial lower and upper bounds for heat kernel  $\phi_t(x)= (4\pi t)^{-\frac{n}{2}} \sum_{k \in \mathbb{Z}^n} e^{-\frac{|x+k|^2}{4 t}} $ (see \eqref{ehk}),  which will play  a vital role in our analysis.
\begin{theorem}\label{crb}
For any $t>0$ and for any $x \in Q_n$,
\begin{equation}\label{eq8}
(4\pi t)^{-\frac{n}{2}}e^{-\frac{|x|^2}{4t}} \leq \phi_t(x) \leq2^n\left(1+\sqrt{\frac{\pi}{t}}\right)^{n}e^{-\frac{|x|^2}{4t}}.
\end{equation}
\end{theorem}
\begin{proof} We first assume that $n=1.$ We shall prove that
 \[(4\pi t)^{-\frac{1}{2}}e^{-\frac{x^2}{4t}} \leq \phi_t(0) e^{-\frac{x^2}{4t}} \leq \phi_t(x) \leq 2 \phi_t(0) e^{-\frac{x^2}{4t}} \leq2\left(1+\sqrt{\frac{\pi}{t}}\right)e^{-\frac{x^2}{4t}}. \]
Note that 
\begin{eqnarray*}
\phi_t(x) & = & (4\pi t)^{-\frac{1}{2}} \sum_{k \in \mathbb{Z}} e^ {-\frac{(x+k)^2}{4 t}}
=(4\pi t)^{-\frac{1}{2}} e^{-\frac{x^2}{4 t}} \sum_{k \in \mathbb{Z}} e^{-\frac{ k x}{2t}} e^{-\frac{ k^2}{4t}} \\
& = & (4\pi t)^{-\frac{1}{2}} e^{-\frac{x^2}{4 t}} \left(1+2\sum_{k=1}^{\infty} e^{-\frac{k^2}{4t}}\cosh\left({\frac{k x}{2t}}\right)\right)\\
& \geq &  (4\pi t)^{-\frac{1}{2}} e^{-\frac{x^2}{4 t}} \left(1+2\sum_{k=1}^{\infty} e^{-\frac{k^2}{4t}}\right)= \phi_t(0) e^{-\frac{x^2}{4 t}}.
\end{eqnarray*}
Let 
$$
g(x):=1+2\sum_{k=1}^{\infty} e^{-\frac{k^2}{4t}}\cosh\left({\frac{k x}{2t}}\right).
$$
Since \(\phi_t(x) = (4\pi t)^{-\frac{1}{2}} e^{-\frac{x^2}{4t}} g(x)\) is an even function, we can consider \(x\) in the interval \([0, \frac{1}{2}]\) for any \(t > 0\). This allows us to bound as follows, for any \(k \geq 1\),
\[ 2 \cosh \left(\frac{ k x}{2t}\right) \leq 2 \cosh \left(\frac{k}{4t}\right) \leq e^{-\frac{k}{4t}} + e^{\frac{k}{4t}} \leq 1 + e^{\frac{k}{4t}}. \]
From this, we conclude that
\[ 2 e^{-\frac{k^2}{4t}} \cosh \left(\frac{kx}{2t}\right) \leq e^{-\frac{k^2}{4t}} + e^{-\frac{k^2}{4t}+\frac{k}{4t}} = e^{-\frac{k^2}{4t}} + e^{-\frac{k(k-1)}{4t}}. \]
Utilizing the inequality above, we have
\[ g(x) \leq 1 + \sum_{k=1}^{\infty} e^{-\frac{k^2}{4t}} + \sum_{k=1}^{\infty} e^{-\frac{k(k-1)}{4t}}. \]
As \(k(k-1) \geq (k-1)^2\), we can simplify further
\[ 
\begin{aligned}
g(x) &\leq \left(1 + \sum_{k=1}^{\infty} e^{-\frac{k^2}{4t}}\right) + \left(\sum_{k=1}^{\infty} e^{-\frac{(k-1)^2}{4t}}\right) \\
&= 2\left(1 + \sum_{k=1}^{\infty} e^{-\frac{k^2}{4t}}\right)= 1 + \left(1 + 2\sum_{k=1}^{\infty}e^{-\frac{k^2}{4t}}\right). 
\end{aligned}
\]
Therefore,
\[ (4\pi t)^{-\frac{1}{2}} g(x) \leq (4\pi t)^{-\frac{1}{2}} + \phi_t(0) \]
and
\[ \phi_t(x) \leq \left((4\pi t)^{-\frac{1}{2}}+ \phi_t(0)\right) e^{-\frac{x^2}{4 t}}. \]
Using the formula
\[ \phi_t(x) = (4\pi t)^{-\frac{1}{2}} e^{-\frac{x^2}{4t}} \left(1+2\sum_{k=1}^{\infty} e^{-\frac{k^2}{4t}}\cosh\left({\frac{kx}{2t}}\right)\right), \]
we can get
\[ \phi_t(0) \geq (4\pi t)^{-\frac{1}{2}}. \]
Hence, we obtain \[\phi_t(x) \leq 2 \phi_t(0) e^{-\frac{x^2}{4t}}.\]
For the last inequality, we express \(\phi_t(x)\) using the Fourier series
\[ \phi_t(x) = 1 + 2 \sum_{k=1}^{\infty} e^{-tk^2} \cos(kx). \]
So,
\[ \phi_t(0) = 1 + 2 \sum_{k=1}^{\infty} e^{-tk^2} = 1 + 2 \varphi(t), \]
where \(\varphi(t) = \sum_{k=1}^{\infty} e^{-tk^2}\).
Comparing \(\varphi\) with an integral, we have
\[ \varphi(t) \leq \int_0^{\infty} e^{-tx^2} \,dx = \frac{1}{2} \sqrt{\frac{\pi}{t}}. \]
So,
\[ \phi_t(0) \leq 1 + \sqrt{\frac{\pi}{t}}. \]
Since
$$\phi_t(x)=(4\pi t)^{-\frac{n}{2}} \sum_{k \in \mathbb{Z}^n}e^ {-\frac{|x+k|^2}{4 t}}=\prod_{i=1}^n \left((4\pi t)^{-\frac{1}{2}}\sum_{k_i \in \mathbb{Z}}e^ {-\frac{(x_i+k_i)^2}{4 t}}\right),$$
our result clearly follows from the one-variable case.
\end{proof}

Since $\mathcal{F}_{WG}^{-1} (e^{-t(|k|^2 + |\xi|^2)})= \mathcal{F}_T^{-1}(e^{-t|k|^2}) \mathcal{F}_{\mathbb R^n}^{-1}(e^{-t|\xi|^2})$,  we obtain

\begin{eqnarray}\label{eq5}
\phi_{t}^{W G}(x, y) & := & \mathcal{F}_{W G}^{-1}\left(e^{-t\left(|k|^{2}+|\xi|^{2}\right)}\right)(x,y)\\
& = &  \mathcal{F}^{-1}_T(e^{-t|k|^2})(x) \mathcal{F}^{-1}_{\mathbb R^n}(e^{-t|\xi|^2})(y)\nonumber\\
& = & (4 \pi t)^{-\left(\frac{n}{2}+\frac{m}{2}\right)} \sum_{k\in\mathbb{Z}^{m}} e^{-\frac{|x+k|^{2}+|y|^{2}}{4t}}.\nonumber
\end{eqnarray}
Utilizing the bound of the heat kernel in the case of torus (\ref{eq8}), the following inequality holds.
\begin{corollary} For any \( t > 0 \) and for any \( (x, y) \in \mathbb{T}^n \times \mathbb{R}^m \),
\begin{equation}\label{eq-r17}
	\left(\frac{1}{4\pi t}\right)^{\frac{n+m}{2}}e^{-\frac{|x|^2+|y|^2}{4t}} \leq \phi^{W G}_t(x, y) \leq2^n\left(1+\sqrt{\frac{\pi}{t}}\right)^{n}\left(\frac{1}{4\pi t}\right)^{\frac{m}{2}}e^{-\frac{|x|^2+|y|^2}{4t}}.
\end{equation}
\end{corollary}

\section{Maximal operator estimates on  $\mathbb T^n$ and  $\mathbb T^n \times \mathbb R^m$}\label{mfe}
In this section, we shall define the maximal  operator on tori and waveguide manifolds and  establish its boundedness for  weighted $L^p$ spaces (see Proposition \ref{GP}  below).   It is at this point we would like to mention that  considerations of particular weights in the main results comes due to Proposition \ref{GP}.   In order to pursue  the aim of this section,  we first briefly recall some definitions and set notations.
For any function \( g \) defined on \( \mathbb{R}^{n} \), the local Hardy-Littlewood maximal operator, denoted by \( \mathcal M_{R}g \), where \( 0 < R < \infty \), is defined as follows
\[ \mathcal M_{R}g(x) = \sup_{0 < r \leq R} \frac{1}{|B(x, r)|} \int_{B(x, r)} |g(y)| \, dy .\]
Here, \( B(x, r) \) represents the ball of radius \( r \) centered at \( x \) and $|E|$ denotes the Lebesgue measure of a subset $E\subset \mathbb R^n.$
Since \( |B(x, r)|=|B(0, 1)| r^n,\) we can rewrite 
\begin{equation}\label{eq-r5}
\mathcal M_{R}g(x) =\frac{1}{|B(0, 1)|} \sup_{0 < r \leq R} \frac{1}{r^n} ( |g| \ast\chi_{B{(0,r)}})(x).
\end{equation}
 In particular, if we regard \( g \) as a \( 1 \)-periodic function on \( \mathbb{R}^{n} \) and set \( R = \frac{1}{2} \), then \( \mathcal M_{\frac{1}{2}}g \) will also be \( 1 \)-periodic.   In fact,  for $k\in\mathbb Z^n,$ we have
\[ \begin{aligned} 
\mathcal M_{\frac{1}{2}}g(k+x) & =\sup _{0<r\leq\frac{1}{2}} \frac{1}{|B(k+x, r)|} \int_{B(k+x, r)} g(y) \, dy \\ & =\sup _{0<r\leq\frac{1}{2}} \frac{1}{|B(x, r)|} \int_{B(x, r)} g(k+y) \, dy =\mathcal M_{1 / 2} g(x).
\end{aligned}\]
Given a function  $f$ on  torus $\mathbb T^n$,   we have a corresponding periodic function on  $\mathbb R^n$,  by setting
\[ f^*(x)= f(e^{2\pi i  x})= f(e^{2 \pi i x_{1}}, \ldots, e^{2 \pi i x_{n}})  \quad   (x\in \mathbb R^n). \]
Note that $f^*$ is $1$-periodic in every co-ordinate.

\begin{definition}\label{HMd} Let \( f \) be an integrable function on  $\mathbb T^n$ and $f^{*}$ be its corresponding $1$-periodic extension. We define the maximal operator on $\mathbb T^n$,  denoted by $\mathcal{M}^T f$, by 
\begin{equation*}\label{mf}
\mathcal{M}^T f(x) = \mathcal M_{\frac{1}{2}} f^{*}(x) \quad for \  all \   x \in Q_n.
\end{equation*}
\end{definition}

Since $\mathbb T^{n}\simeq Q_{n},$ we can represent the set $\mathbb{T}^n \times\mathbb{R}^m$ as $Q_{n} \times \mathbb{R}^{m}$. Given a function $f$ defined on $\mathbb{T}^n \times\mathbb{R}^m$, similarly,  we can associate it with a corresponding partially periodic function on $\mathbb{R}^{n+m}$ by the extension
\begin{equation}\label{eq4}
f^{*}(x,y)=f\left(e^{2 \pi i x}, y\right),
\end{equation}
where $e^{2 \pi i x}=\left(e^{2 \pi i x_{1}}, \ldots, e^{2 \pi i x_{n}}\right)$. Note that $f^{*}$ is $1$-periodic in the $x$ variable.

\begin{definition}Let $f$ be a locally integrable function on the waveguide $\mathbb{T}^n \times\mathbb{R}^m$, and let $f^{*}$ be its corresponding periodic  extension (see \eqref{eq4}) on $\mathbb{R}^{n+m}$. 
We  define the maximal operator on the waveguide $\mathbb{T}^n \times\mathbb{R}^m,$  denoted by  $\mathcal M^{WG},$ by
$$
\mathcal M^{W G} f(x,y)=\mathcal M_{\frac{1}{2}} f^{*}(x,y) \quad for \  all \  (x, y) \in Q_n \times\mathbb{R}^m,
$$
where $\mathcal M_{\frac{1}{2}} f^{*}$ represents the local Hardy-Littlewood maximal operator on $\mathbb R^{n+m}$. 
\end{definition}

\begin{remark} \
\begin{itemize}
\item[-]
Since \( \mathcal M_{\frac{1}{2}} f^{*} \) is periodic, the periodic extension of \( \mathcal{M}^T f \) on \( \mathbb{R}^n \) is \( \mathcal M_{\frac{1}{2}} f^{*} \). In other words, \( (\mathcal{M}^T f)^{*} = \mathcal M_{\frac{1}{2}} f^{*} \).
\item[-] Let \(f\) be any measurable function on \(\mathbb{T}^n\) and \(f^*\) be its \(1\)-periodic extension. Then, for any \(x \in Q_n\), every ball \(B(x, r)\), with \(0 < r \leq \frac{1}{2}\), is contained within one period of \(f^*\) without any repetition. Since our main result concerns the pointwise limit, taking \(R = \frac{1}{2}\) is sufficient for this scenario.
\end{itemize}
\end{remark}
 We are now ready to state sharp boundedness result of  $\mathcal{M}^T$ and $\mathcal{M}^{WG}$  on weighted Lebesgue spaces.  Specifically,  we state the following proposition.
\begin{proposition}\label{GP}
Let  \(1 < p < \infty\) and $v$ be a strictly positive weight. Then
\begin{enumerate}
\item \label{prop2} 
  The weight  $v^{-\frac{1}{p}} \in L^{p^{\prime}}(\mathbb T^n)$ if and only if  there exists a weight \(u\) such that the operator \(\mathcal{M}^{T}\) maps \(L_v^p(\mathbb T^n)\) into \(L_u^p(\mathbb T^n)\)  with the norm inequality
\[ \| \mathcal{M}^T    f\|_{L^{p}_u(\mathbb{T}^n)} \lesssim \|f\|_{L^p_v(\mathbb{T}^n)}.\]
And  the weight    $v^{-1} \in L^{\infty}(\mathbb T^n)$ if and only if  there exists a weight \(u\) such that the operator \(\mathcal{M}^T\) maps \(L_v^1(\mathbb T^n)\) into weak \(L_u^1(\mathbb T^n)\).

\item \label{prop-r2}  The weight   $v^{-\frac{1}{p}} \in  L_{loc}^{p^{\prime}}(\mathbb{T}^n \times\mathbb{R}^m)$ if and only if  there exists a weight \(u\) such that the operator \(\mathcal{M}^{WG}\) maps \(L_v^p(\mathbb{T}^n \times\mathbb{R}^m)\) into \(L_u^p(\mathbb{T}^n \times\mathbb{R}^m)\)  with the norm inequality
	\[ \| \mathcal{M}^{WG}    f\|_{L^{p}_u(\mathbb{T}^n \times\mathbb{R}^m)} \lesssim \|f\|_{L^p_v(\mathbb{T}^n \times\mathbb{R}^m)}.\]
And the weight  $v^{-1} \in L_{loc}^{\infty}(\mathbb{T}^n \times\mathbb{R}^m)$ if and only if there exists a weight \(u\) such that the operator \(\mathcal{M}^{WG}\) maps \(L_v^1(\mathbb{T}^n \times\mathbb{R}^m)\) into weak \(L_u^1(\mathbb{T}^n \times\mathbb{R}^m)\). 
\end{enumerate}
\end{proposition}

In order to state Proposition \ref{vvi},   we set some notations.
Let \( f=\left\{f_j\right\} \) be a vector-valued function on \( \mathbb{R}^n \) and \( 1 < p < \infty \).   Denote
\begin{itemize}
\item  $\overline{\mathcal{M}^{p}_{R} f}=\left(\sum_{j}\left|\mathcal M_{R} f_{j}\right|^{p}\right)^{1 / p},$ \quad $ \overline{f^p} = \left(\sum_{j}\left| f_{j}\right|^{p}\right)^{1 / p}$
\item $ \overline{\mathcal{M}^{p} f} = \left(\sum_{j}\left|\mathcal{M} f_{j}\right|^{p}\right)^{1 / p}, $
\end{itemize}
where   Hardy-Littlewood maximal operator $\mathcal{M}f_j$ is given by
\[
\mathcal{M}f_j(x)=\sup_{r>0} \frac{1}{|B(x, r)|} \int_{B(x, r)} |f_j(y)| \, dy.
\]
\begin{proposition}[see Theorem 1 in \cite{MR0284802}]\label{vvi}
	For $1<p,q<\infty,$  we have the vector-valued strong type \( (q, q) \) inequality
	\[ \left\|\overline{\mathcal{M}^{p} f}\right\|_{L^q(\mathbb{R}^n)}\leq C_{p,q}\left\|\overline{f^{p}}\right\|_{L^q(\mathbb{R}^n)}\]
	and the vector-valued weak type \( (1,1) \) inequality
	\[ \left|\left\{x:\overline{\mathcal{M}^{p} f}(x)>\lambda\right\}\right|\leq \frac{C_p}{\lambda}\left\|\overline{f^p}\right\|_{L^1(\mathbb{R}^n)},\quad (\lambda>0) . \]
\end{proposition}
Before we jump into the proof of  Proposition \ref{GP}, we shall first briefly indicate the key ideas in the following remark.
\begin{remark}\label{PSGP} The following key points are in order to establish the boundedness  of $\mathcal M^{T}$ and $\mathcal{M}^{WG}$.
\begin{itemize}
\item[-] We dominate  $\mathcal{M}^{T}f$ by certain local  maximal operator of its periodic extension ( specifically by
$\mathcal{M}_{\frac{1}{2}}f^{*}_{\tilde{Q}_n}$). 
\item[-] As a consequence,   we are left to show \(L^p_v\)-boundedness of the local maximal operator on some compact subsets in \(\mathbb{R}^n\) (see \eqref{c}).
\item[-] To this end,  we invoke Lemma \ref{lm-r1}. 
\item[-] In order to verify the hypothesis of Lemma \ref{lm-r1},  we invoke  Kolmogorov's inequality (i.e.  Lemma \ref{ki}).
\item[-] In order to verify hypothesis of Lemma \ref{ki},  we invoke weak type $(1,1)$ boundedness of vector valued maximal operator (i.e.  Proposition \ref{vvi}).  This essentially settles the proof on tori.
\item[-] In the case of waveguide $\mathbb T^n\times\mathbb R^m$, which is not compact, we decompose it into disjoint countable annuli \( A_k \) for \( k \geq 0 \). We then decompose (not disjoint) \(\tilde{Q}_n \times \mathbb{R}^m\) into countable balls \( B_k \) (see Figure \ref{fig:graph}). For each \( k \), we substitute \( A_k \) with  \(Q_n\) and \( B_k \) with \(\tilde{Q}_n\). By following the same steps mentioned in the torus case, we can obtain the result by combining these individual components.
\item[-] The converse part follows from the Kolmogorov inequality (i.e.  Lemma \ref{ki}) and a subsequent duality argument.

\end{itemize}
\end{remark}

We will now move on to prove   Proposition \ref{GP}.

\begin{proof}[\textbf{Proof of Proposition \ref{GP},\eqref{prop2}}] 
Assume that $v^{-\frac{1}{p}} \in L^{p^{\prime}}(\mathbb T^n)$ for $1<p<\infty.$ Let $f\in L^p_v(\mathbb{T}^n)$ and $f^*$ be the periodic extention of $f.$
Define the set
	$$G=\left\{\left(x_{1}, \ldots,  x_{n}\right): x_{i} \in\{-1,0,1\}\text{ for } 1 \leq i \leq n\right\}.$$
	Denote $\tilde{Q}_{n}:=\bigcup_{a \in G} \{Q_{n}+a \}\simeq[-\frac{3}{2},\frac{3}{2}]^{n}$ and  put  \( f_{\tilde{Q}_n}^{*}(x) = f^{*}(x) \chi_{\tilde{Q}_n}(x)\). We observe that
	\begin{equation}\label{eq-r1}
		\int_{\tilde{Q}_n} f^{*}_{\tilde{Q}_n} \, dx = 3^n  \int_{\mathbb T^n} f(x) \, dx .
	\end{equation}
Decompose \( f^{*}\) as \( f^{*} = f_{\tilde{Q}_n}^{*} + f_{(\tilde{Q}_n)^c}^{*} \).  Note that $B(x,r)\subset \tilde{Q}_n$ for $x\in Q_n$ and $0<r\leq\frac{1}{2}.$ Thus,  we have 
 \( \mathcal M_{\frac{1}{2}}f^{*}_{(\tilde{Q}_n)^c}(x) = 0 \) for all \( x \in Q_n \). Hence, 
 \begin{equation}\label{lmf}
 \mathcal M^T f(x) = \mathcal M_{\frac{1}{2}} f^{*}(x) \leqslant \mathcal M_{\frac{1}{2}} f_{\tilde{Q}_n}^{*}(x) \quad for \ all \   x \in \mathbb{T}^n.
 \end{equation}
 Note that, for any pisitive weight $u$ on $\mathbb{T}^n,$ we have 
\begin{eqnarray}\label{u1}
\|\mathcal{M}^T f\|_{L_u^p(\mathbb T^n)} & = & \|\mathcal{M}_{\frac{1}{2}} f^*\|_{L_u^p(Q_n)}.
\end{eqnarray} 
 We now claim that  there exists a positive weight on $u$ on $Q_n$ such that  the following inequality holds:
 \begin{eqnarray}\label{c}
 \|\mathcal{M}_{\frac{1}{2}} f^*\|_{L_{u}^p(Q_n)} \lesssim  \|f^*\|_{L^p_{v^*}(\tilde{Q}_n)},
 \end{eqnarray}
 where $v^*$ is the periodic extension of $v$ on $\tilde{Q}_n.$
 By combining \eqref{eq-r1}, \eqref{u1} and \eqref{c},   we can obtain the desired inequality
\[ \| \mathcal{M}^T    f\|_{L^{p}_u(\mathbb{T}^n)} \lesssim \|f\|_{L^p_v(\mathbb{T}^n)}.\]
Now we must prove our claim \eqref{c}.  In fact,  in view of Lemma \ref{lm-r1},  it is sufficient to prove,    for $0<s<p< \infty,$
 \begin{equation}\label{eq-r3}
\left\|\overline{\mathcal{M}^{p}_{\frac{1}{2}} f^*}\right\|_{L^{s}(Q_n)} \lesssim \left(\sum_{j}\left\|f^*_{j}\right\|_{L_{v^*}^{p}(\tilde{Q}_n)}^{p}\right)^{\frac{1}{p}},
\end{equation}
where $f^*=\{f^*_j\}$ is vector-valued $1$-periodic function. Note that, for $f^*_{\tilde{Q}_n}=\{(f^*_{\tilde{Q}_n})_{j}\} ,$ we have
\begin{eqnarray}\label{uw}
  \left\|\overline{\mathcal{M}^{p}_{\frac{1}{2}} f^*}\right\|_{L^{s}\left(Q_n\right)} 
   \leq 
	\left\|\overline{\mathcal{M}^{p}_{\frac{1}{2}} f^*_{\tilde{Q}_n}}\right\|_{L^{s}\left(Q_n\right)}
	\leq \left\|\overline{\mathcal{M}^{p} f^*_{\tilde{Q}_n}}\right\|_{L^{s}\left(Q_n\right)}.
\end{eqnarray}
In view of   Proposition \ref{vvi},  we notice that 
\[\overline{ f^{*\ p}_{\tilde{Q}_n}} \mapsto \overline{\mathcal{M}^{p} f^*_{\tilde{Q}_n}} \]
is a  weak type $(1,1)$ operator.  Now,  by \eqref{uw} and Lemma \ref{ki} and H\"older's inequality,  we obtain 
 \begin{eqnarray*}
   \left\|\overline{\mathcal{M}^{p}_{\frac{1}{2}} f^*}\right\|_{L^{s}\left(Q_n\right)} 
	&\lesssim& \left\|\overline{ f^{*\ p}_{\tilde{Q}_n}}\right\|_{L^{1}({\tilde{Q}_n})}\\
	& \lesssim & \left(\int_{{\tilde{Q}_n}} \sum_{j}\left|f^*_{j}(x)\right|^{p} v^*(x) d x\right)^{\frac{1}{p}}\left(\int_{{\tilde{Q}_n}} (v^*)^{-\frac{p^{\prime}}{p}}(x) d x\right)^{\frac{1}{p'}}.
 \end{eqnarray*}
Since $v^{-\frac{1}{p}}\in L^{p'}(\mathbb T^n),$  the above inequality proves  \eqref{c}.
For the case \( p = 1 \), we utilize the weak \( (1,1) \) bound of the standard Hardy-Littlewood maximal operator. For $\lambda\geq 0,$ we have
\begin{eqnarray}\label{L1}
	\left|\{x\in \mathbb{T}^n: \mathcal{M}^T f(x)>\lambda\}\right|\nonumber
	&\leq&\left|\{x\in Q_n: \mathcal{M}_{\frac{1}{2}}f^*_{\tilde{Q}_n}(x)>\lambda\}\right|\\\nonumber
	&\leq&\left|\{x\in Q_n: \mathcal{M}f^*_{\tilde{Q}_n}(x)>\lambda\}\right|\\
	&\leq& \frac{C}{\lambda}\left\|f^*_{\tilde{Q}_n}\right\|_{L^1({\tilde{Q}_n})}=\frac{C}{\lambda} \int_{\tilde{Q}_n}|f^*(x)| d x \\\nonumber
	& \leq& \frac{C}{\lambda}\left(\int_{\tilde{Q}_n}|f^*(x)| v^*(x) d x\right)\left\|(v^*)^{-1} \chi_{\tilde{Q}_n}\right\|_{L^{\infty}({\tilde{Q}_n})} \\\nonumber
	& \leq &\frac{C_{n, v}}{\lambda} \int_{\tilde{Q}_n}|f^*(x)| v^*(x) d x =\frac{C_{k, v}}{\lambda}\|f\|_{L^1_v(\mathbb{T}^n)},\nonumber
\end{eqnarray}
where $C_{n, v}$ is a constant depending on $v$ and $n.$ Hence, the weight \( u(x) = \chi_{Q_n}(x) \) provides weak boundedness for the case \( p = 1 \).

Now we shall  prove the converse part.   Assume that the operator \(\mathcal{M}^T\) maps \(L_v^p(\mathbb{T}^n)\) into weak \(L_u^p(\mathbb{T}^n)\) for some weight \(u\) and for \(1 \leq p < \infty\).   Let us define the operator
\begin{equation*}
M_{r}f(x) = \frac{1}{r^n} \left( |f| \ast \chi_{B(0,r)} \right)(x)\quad (0<r\leq\frac{1}{2}).
\end{equation*}
Since $\mathcal{M}^Tf(x) = \frac{1}{|B(0, 1)|} \sup_{0 < r \leq \frac{1}{2}} M_{r}f(x)$ (see \eqref{eq-r5}), the operator \(M_{r}\) also maps \(L_v^p(\mathbb{T}^n)\) into the weak \(L_u^p(\mathbb{T}^n)\) space.
Let $0<R\leq \frac{1}{4}$ and \( x_0 \in \mathbb{T}^n \). Then \( B\left(x_0, R\right) \subset B(x, 2R) \) for \( x \in B\left(x_0, R\right) \). For any nonnegative \( f \in L^p_v(\mathbb{T}^n) \), we have
\[
M_{2R} f(x) = \frac{1}{(2R)^n}\int_{B(x, 2R)} f(y) \, dy \geq \frac{1}{(2R)^n}\int_{B\left(x_0, R\right)} f(y) \, dy 
\]
for all $x \in B(x_0, R).$ Hence,
\[	u\left(B\left(x_0, R\right)\right)^{\frac{1}{s}}\frac{1}{(2R)^n}\left( \int_{B\left(x_0, R\right)} f(x) d x\right)\leq\left(\int_{B\left(x_0, R\right)} M_{2R} f(x)^s u(x) d x\right)^{\frac{1}{s}}.\]
Using Lemma \ref{ki} for \( s < p \), we have
$$
\left(\int_{B\left(x_0, R\right)} M_{2R} f(x)^s u(x) d x\right)^{\frac{1}{s}}
\lesssim \,u\left(B\left(x_0, R\right)\right)^{\frac{1}{s}-\frac{1}{p}}\left(\int_{\mathbb{T}^n} f^p(x) v(x) d x\right)^{\frac{1}{p}}.
$$
Choosing $f=g v^{-\frac{1}{p}}$ such that \( g \in L^p(\mathbb{T}^n) \), we can rewrite
$$
\left( \int_{B\left(x_0, R\right)} g(x) v^{-\frac{1}{p}}(x) d x\right) \lesssim \, u\left(B\left(x_0, R\right)\right)^{-\frac{1}{p}}\left(\int_{\mathbb{T}^n} g^p(x) d x\right)^{\frac{1}{p}} .
$$
By duality we can conclude that $v^{-\frac{1}{p}}$ belongs to $L^{p^{\prime}}\left(B\left(x_0, R\right)\right)$. Since \(\mathbb{T}^n\) is compact, we can conclude our claim.
\end{proof}

\begin{remark}\label{rm-r1}
It should be noted that in the proof of \eqref{c}, there is no specific  uniqueness assigned to the sets $Q_n$ and ${\tilde{Q}_n}$.  In fact,  we can claim the same conclusion for more generic sets.  Specifically,   we can substitute $Q_n$ with annulus $A_k$ and ${\tilde{Q}_n}$ with ball $B_k$ for any $k\geq0$ (the precise definitions of $A_k$ and $B_k$ will be provided in the proof of Proposition \ref{GP},(\ref{prop-r2}). Here, $B_k$ must contain $A_k$ and satisfy the condition
$$
\mathcal M_{\frac{1}{2}} f_{B_{k}^{c}}^{*}(x, y)=0\text { for all }(x, y) \in A_{k}.
$$
Subsequently, by following a parallel procedure to the proof of \eqref{c}, there exist a positive weight $u_k$ such that
\begin{eqnarray}
	\|\mathcal{M}_{\frac{1}{2}} f^*\|_{L_{u_k}^p(A_k)} \lesssim  \|f^*\|_{L^p_{v^*}(B_k)},
\end{eqnarray}
where $f^*$ and $v^*$ are well-defined on $B_k$.
\end{remark}

\begin{proof}[\textbf{Proof of Proposition \ref{GP},\eqref{prop-r2}}]
 Assume that  $1<p<\infty$.  For $k\geq 1,$ let define the sets
	$$
	\begin{aligned}
		A_{0}=&\left\{(x, y) \in Q_{n} \times \mathbb{R}^{m}:|(x, y)|<\frac{1}{2}\right\},\\ 
		A_{k}=&\left\{(x, y) \in Q_{n} \times \mathbb{R}^{m}: k-\frac{1}{2} \leq|(x, y)|<k+\frac{1}{2}\right\},\\
		B_{0}=&\left\{(x, y) \in \tilde{Q}_{n} \times \mathbb{R}^{m}:|(x, y)|<\frac{3}{2}\right\},\\
		B_{k}=&\left\{(x, y) \in \tilde{Q}_{n} \times \mathbb{R}^{m}:|(x, y)|<k+\frac{3}{2}\right\}.
	\end{aligned}
	$$
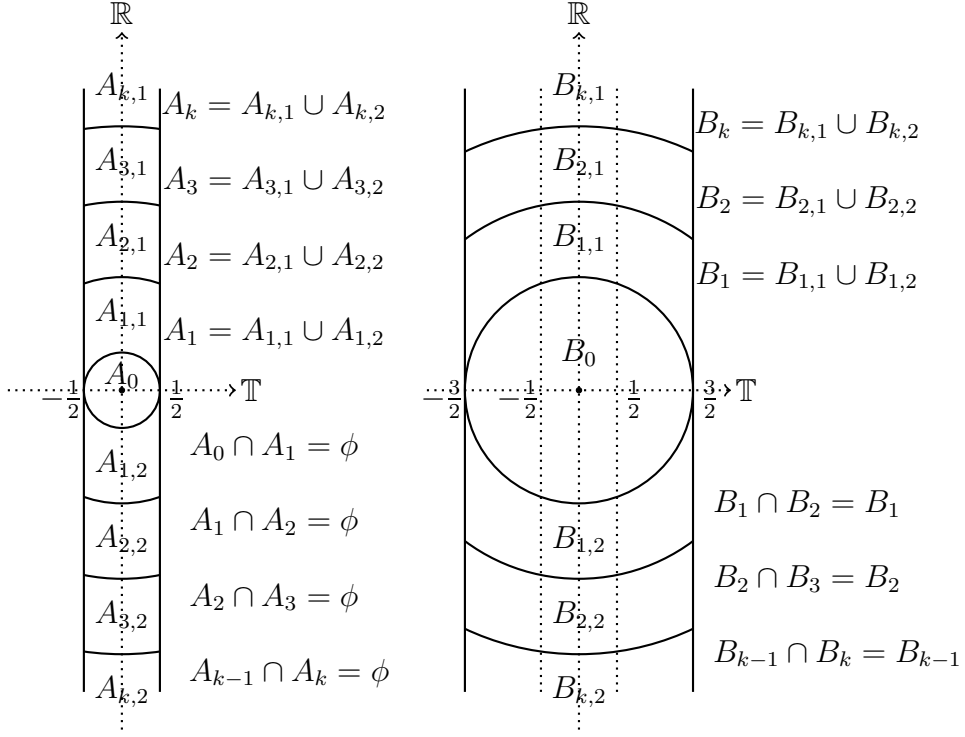
\begin{figure}[h]
	\centering
	\begin{tikzpicture}[scale=1]
		\draw[thick,dotted, ->] (-1.5,0) -- (1.5,0);
		\draw[thick,dotted, ->] (0,-4.5) -- (0,4.75);

		\filldraw(0,0) circle (1pt);
		
		\node at (1.7,0) {$\mathbb {T}$};
		\node at (0,5) {$\mathbb R$};
		\node at (0.72,-0.12) {$\frac{1}{2}$};
		\node at (-0.8,-0.12) {$-\frac{1}{2}$};
		\node at (0,0.2) {$A_0$};
		\node at (0,1) {$A_{1,1}$};
		\node at (0,2) {$A_{2,1}$};
		\node at (0,3) {$A_{3,1}$};
		\node at (0,4) {$A_{k,1}$};
		\node at (0,-1) {$A_{1,2}$};
		\node at (0,-2) {$A_{2,2}$};
		\node at (0,-3) {$A_{3,2}$};
		\node at (0,-4) {$A_{k,2}$};
		\node at (2,0.75) {$A_1=A_{1,1}\cup A_{1,2}$};
		\node at (2,1.75) {$A_2=A_{2,1}\cup A_{2,2}$};
		\node at (2,2.75) {$A_3=A_{3,1}\cup A_{3,2}$};
		\node at (2,3.75) {$A_k=A_{k,1}\cup A_{k,2}$};
		\node at (2,-0.75) {$A_0\cap A_1=\phi$};
		\node at (2,-1.75) {$A_1\cap A_2=\phi$};
		\node at (2,-2.75) {$A_2\cap A_3=\phi$};
		\node at (2.2,-3.75) {$A_{k-1}\cap A_k=\phi$};

		\draw[thick] (-0.5, -4) -- (-0.5, 4);
		\draw[thick] (0.5, -4) -- (0.5, 4);
		
		\draw[thick, black] (0,0) circle [radius=0.5];
		
		\clip (-.5,-4) rectangle (.5,4);
		\draw[thick] (0,0) circle [radius=1.5];
		
		\clip (-.5,-4) rectangle (.5,4);
		\draw[thick] (0,0) circle [radius=2.5];
		
		\clip (-0.5,-4) rectangle (0.5,4);
		\draw[thick] (0,0) circle [radius=3.5];
		
	\end{tikzpicture}
	\begin{tikzpicture}[scale=1]
		\draw[thick,dotted, ->] (-2,0) -- (2,0);
		\draw[thick,dotted, ->] (0,-4.5) -- (0,4.75);

		\filldraw(0,0) circle (1pt);
		
		\node at (2.2,0) {$\mathbb {T}$};
		\node at (0,5) {$\mathbb R$};
		\node at (0.72,-0.12) {$\frac{1}{2}$};
		\node at (-0.8,-0.12) {$-\frac{1}{2}$};
		\node at (1.72,-0.12) {$\frac{3}{2}$};
		\node at (-1.8,-0.12) {$-\frac{3}{2}$};
		\node at (0,.5) {$B_0$};
		\node at (0,2) {$B_{1,1}$};
		\node at (0,3) {$B_{2,1}$};
		\node at (0,4) {$B_{k,1}$};
		\node at (0,-2) {$B_{1,2}$};
		\node at (0,-3) {$B_{2,2}$};
		\node at (0,-4) {$B_{k,2}$};
		\node at (3,1.5) {$B_1=B_{1,1}\cup B_{1,2}$};
		\node at (3,2.5) {$B_2=B_{2,1}\cup B_{2,2}$};
		\node at (3,3.5) {$B_k=B_{k,1}\cup B_{k,2}$};
		\node at (3,-1.5) {$B_1\cap B_2=B_1$};
		\node at (3,-2.5) {$B_2\cap B_3=B_2$};
		\node at (3.4,-3.5) {$B_{k-1}\cap B_k=B_{k-1}$};

		\draw[thick,dotted] (-0.5, -4) -- (-0.5, 4);
		\draw[thick,dotted] (0.5, -4) -- (0.5, 4);
		\draw[thick] (-1.5, -4) -- (-1.5, 4);
		\draw[thick] (1.5, -4) -- (1.5, 4);
		
		\draw[thick] (0,0) circle [radius=1.5];
		
		\clip (-1.5,-4) rectangle (1.5,4);
		\draw[thick] (0,0) circle [radius=3.5];
		
		\clip (-1.5,-3) rectangle (1.5,3);
		\draw[thick] (0,0) circle [radius=2.5];
		
	\end{tikzpicture}
	\caption{We have drawn the graphs of $A_k$ and $B_k$ in the case of $\mathbb T\times\mathbb R$. On the left side, we have shown the disjoint annuli $A_k$, while on the right side, we have illustrated the balls $B_k$ containing $A_k$.}
	\label{fig:graph}
\end{figure}
Note that $A_0$ is the open ball with a radius of $\frac{1}{2}$, while $A_k$ are disjoint annulus in $Q_n \times \mathbb R^m$. Meanwhile, $B_0$ and $B_k$ are open balls with radius of $\frac{3}{2}$ and $k+\frac{3}{2}$ containing $A_0$ and $A_k,$ respectively. Consequently, we have $\bigcup_{k=0}^\infty A_k = Q_{n} \times \mathbb{R}^{m}$ and $\bigcup_{k=0}^\infty B_k = \tilde{Q}_{n} \times \mathbb{R}^{m}$. We refer to Figure \ref{fig:graph}.
For each $k\geq0$, we decompose $f^{*}=  f^*\chi_{B_{k}}+f^*\chi_{B^c_{k}} =   f_{B_{k}}^{*}+f_{B^c_{k}}^{*}$.  Since $A_k \subset B_k,$ we  have
	$$
	\mathcal M_{\frac{1}{2}} f_{B_{k}^{c}}^{*}(x, y)=0\ for\ all\ (x, y) \in A_{k}.
	$$
	Hence, we have 
	$$\mathcal M^{W G} f(x, y)=\mathcal M_{\frac{1}{2}} f^{*}(x, y) \leq \mathcal M_{\frac{1}{2}} f_{B_{k}}^{*}(x, y)\  for\ all\ (x, y) \in A_{k}.$$
	Note that, for any positive weight $u_k$ on $A_k,$ we have 
	\begin{eqnarray}\label{}
		\|\mathcal{M}^{WG} f\|_{L^p_{u_k}(A_k)} & = & \|\mathcal{M}_{\frac{1}{2}} f^*\|_{L^p_{u_k}(A_k)}.
	\end{eqnarray} 
	Let  $v^*$ be the periodic extension of $v$ on ${\tilde{Q}_n\times \mathbb R^m}.$ For each $k\geq0,$ putting all of these setups $A_k$, $B_k$, $f^*$, and $v^*$ into Remark \ref{rm-r1}, we conclude that there exists a weight $u_k$ on $A_k$ such that the following inequality holds
	\begin{eqnarray}
		\|\mathcal{M}_{\frac{1}{2}} f^*\|_{L_{u_k}^p(A_k)} \leq C_k \|f^*\|_{L^p_{v^*}(B_k)}\leq C_k \|f^*\|_{L^p_{v^*}(\tilde{Q}_n\times\mathbb R^m)}.
	\end{eqnarray}
	Since the family of weights $u_k$, each one with support in $A_k$, can be combined together, we can express $u(x,y)$ as a weight on $Q_{n} \times \mathbb{R}^{m}$ as follows
	$$u(x,y)=\sum_{k\geq 0 }\frac{1}{\left(2^k C_{k}\right)^p} u_k(x,y) \chi_{A_k}(x,y).$$
	Hence, we conclude that
	\begin{eqnarray}\label{eq-r15}
		\|\mathcal{M}_{\frac{1}{2}} f^*\|_{L_{u}^p(Q_{n} \times \mathbb{R}^{m})} \lesssim \|f^*\|_{L^p_{v^*}(\tilde{Q}_n\times\mathbb R^m)}.
	\end{eqnarray}
Similar to \eqref{eq-r1}, we can write the analogous equation for the waveguide case
\begin{equation}\label{eq-r14}
	\int_{\tilde{Q}_{n} \times \mathbb{R}^{m}} f^{*}(x,y)dx\,dy=3^{n} \int_{\mathbb T^{n} \times \mathbb R^{m}} f(x,y) d x\, dy.
\end{equation}
By combining \eqref{eq-r15} and \eqref{eq-r14},   we get the desired inequality
\[ \| \mathcal{M}^{WG}    f\|_{L^{p}_u(\mathbb{T}^n\times\mathbb R^m)} \lesssim \|f\|_{L^p_v(\mathbb{T}^n\times\mathbb R^m)}.\]
For the case \( p = 1 \), we will use the weak type \( (1,1) \) boundedness of the Hardy-Littlewood maximal operator. For \(\lambda \geq 0\), by a similar estimate as in \eqref{L1}, replacing \( Q_n \) with \( A_k \) and \(\tilde{Q}_n\) with \( B_k \) for each $k\geq0,$ we obtain 
\[
\begin{aligned}
	\left|\{(x,y)\in A_k: \mathcal{M}_{\frac{1}{2}}f^*(x,y)>\lambda\}\right|
	& \leq \frac{C_{k, v}}{\lambda} \int_{B_{k}}|f^*(x,y)| v^*(x,y) d xdy \\
	&=\frac{C_{k, v}}{\lambda}\|f^*\|_{L^1_{v^*}(\tilde{Q}_{n} \times \mathbb{R}^{m})}.
\end{aligned}
\]
Hence, the weight $u^*(x,y)=\sum_{k\geq 0 }\frac{1}{\left(2^k C_{k,v}\right)^p} \chi_{A_k}(x,y)$ provides 
	\begin{equation}\label{eq-r16}
	\left|\{(x,y)\in\tilde{Q}_{n} \times \mathbb{R}^{m} : \mathcal{M}_{\frac{1}{2}}f^*(x,y)>\lambda\}\right| \lesssim\frac{1}{\lambda}\|f^*\|_{L^1_{v^*}(\tilde{Q}_n\times\mathbb R^m)}.
	\end{equation}
By combining  \eqref{eq-r14} and \eqref{eq-r16},   we get the desired weak boundedness.
On the other hand, by proceeding in exactly the same manner as in the converse part of Proposition \ref{GP},\eqref{prop2}, we can conclude that \(v^{-\frac{1}{p}} \in L_{loc}^{p^{\prime}}(\mathbb{T}^n\times\mathbb{R}^m)\) for \(1 < p < \infty\), and \(v^{-1} \in L_{loc}^{\infty}(\mathbb{T}^n\times\mathbb{R}^m)\) when $p=1$.
\end{proof}

\begin{remark}
The strong type \((p,p)\) boundedness in Proposition \ref{GP} does not hold for \(p=1\).  For example, consider the weight \(v\) as constant. And \(f(x) = |x|^{-n} \left(\ln \frac{e}{x}\right)^{-2} \chi_{B(0,r)} \in L^1(\mathbb{T}^n)\) for some \(0 < r < \frac{1}{2}\). However, \(|x|^{-n} \left(\ln \frac{e}{x}\right)^{-1} \lesssim \mathcal{M}^{T} f(x)\), which is not locally integrable near $0$. Since \(f\) is supported in a small ball, a similar example works for \(\mathbb{T}^n \times \mathbb{R}^m\).  See also  Remark 2.3 in \cite{torreaTAMS}.
\end{remark}

\section{Proof of  Main Results}\label{mrs}
\subsection{Characterization of  weight classes  $D^{T}_p$ and $D^{WG}_p$ }
In this section,  we shall characterize weight classes  $D^{T}_p$ and $D^{WG}_p$,  defined in the introduction  (see Definitions \ref{wct} and \ref{def-r1}).
\begin{proposition}[characterization of  $D^{T}_p$]\label{prop1}
Let $v$ be a pisitive weight in $\mathbb{T}^n$and $ 1 \leq p<\infty.$
Then the  following statements are equivalent:
\begin{enumerate}
\item \label{1} The weight $v\in D_p^{T}$ for some $ 1 \leq p<\infty.$
\item \label{2} There exists $t_0>0$ and a weight $u$ such that the operator $f \rightarrow \phi_{t_0} * f$ maps $L_v^p(\mathbb T^n)$ into $L_u^p(\mathbb T^n)$ for $p>1,$ and maps $L_v^1(\mathbb T^n)$ into weak $L_u^1(\mathbb T^n)$ when $p=1$.
\item \label{3} There exists $t_0>0$ such that $\phi_{t_0} * f(x)$ is finite almost everywhere for all $f \in L_v^p(\mathbb T^n)$.
\end{enumerate}
\end{proposition}
\begin{remark}
In each of the statements in Proposition \ref{prop1}, the constants \(t_0\) are not identical, but they are proportional. This holds true for the subsequent results as well.
\end{remark}
\begin{proof} We shall prove that \eqref{1} implies \eqref{2}. 
By H{\"o}lder's inequality, for any $t>0,$ we have 
\begin{eqnarray}\label{eq1}
\| \phi_t \ast f\|^p_{L^p_u(\mathbb T^n)}
 & \leq &  \int_{\mathbb T^n} \left\|\phi_{t}(x-\cdot) v^{-\frac{1}{p}}(\cdot)\right\|_{L^{p^{\prime}}(\mathbb T^n)}^{p} \| v^{\frac{1}{p}} f\|^p_{L^p(\mathbb T^n)} u(x) dx\nonumber\\
 & = & \|  f\|^p_{L_v^p(\mathbb T^n)}  \int_{\mathbb T^n} g^t(x) u(x) dx,
\end{eqnarray}
where $ g^t(x)= \left\|\phi_{t}(x-\cdot) v^{-\frac{1}{p}}(\cdot)\right\|_{L^{p^{\prime}}(\mathbb T^n)}^{p}.$ We claim that  there exists $t_0>0$ such that $g^{t_0}(x)$ is finite for all $x$.  Now choosing \(u \in L^{1}(\mathbb T^n)\) such that $ g^{t_0} \cdot u$ is in \(L^{1}(\mathbb T^n)\).  Indeed,  if the function $g^t$ is unbounded and the set $\{x: g^t(x) \geq 1\}$ has nonzero measure, we define the function $u$ as follows: for $x$ in $\{x: g^{t_0}(x) \geq 1\}$, set $u(x) = \frac{1}{g^{t_0}(x)}$, and for all other points, set $u(x) = 0$. If the set $\{x: g^{t_0}(x) \geq 1\}$ has zero measure, or if $g^{t_0}$ is bounded, the natural choice for $u$ is a constant function.    Thus,   we have  $\|\phi_{t_0}\ast f\|_{L^p_u(\mathbb T^n)}  \lesssim \| f\|_{L^p_v(\mathbb T^n)}.$

Now, we must prove our claim.
    For a given  $x\in Q^n$ consider the ball \(B_{x}=\left \{y \in Q^n:|x-y| \leq|x| \right \}\). Consequently, for \(t>0\), we have
\begin{eqnarray}\label{ds0}
&\left\|\phi_{t}(x-\cdot) v^{-\frac{1}{p}}(\cdot)\right\|_{L^{p^{\prime}}(\mathbb T^n)}\\
&\leq\left\|\chi_{B_{x}}(\cdot) \phi_{t}(x-\cdot) v^{-\frac{1}{p}}(\cdot)\right\|_{L^{p^{\prime}}(\mathbb T^n)}+\left\|\chi_{B_{x}^{c}}(\cdot) \phi_{t}(x-\cdot) v^{-\frac{1}{p}}(\cdot)\right\|_{L^{p^{\prime}}(\mathbb T^n)}.\nonumber
\end{eqnarray}
If \(|x-y| \leq |x|\), then \(|y| \leq |x-y|+|x| \leq 2|x|\).   By (\ref{eq8}), for  \(|x-y| \leq |x|\),   we have
\begin{eqnarray}\label{ds1}
\phi_{t}(x-y) &   \leq  & C_n =  C_n e^{\frac{1}{4t}|x|^{2}} e^{-\frac{1}{4t}|x|^2} \leq C_n e^{\frac{1}{4t}|x|^{2}} e^{-\frac{1}{4t}\left(\frac{|y|}{2}\right)^{2}}  \nonumber \\
&\leq &  C_n(4\pi t)^{\frac{n}{2}} e^{\frac{1}{4t}|x|^{2}} \sum_{k \in \mathbb{Z}^n}e^{-\frac{1}{4t}\left(\frac{|y+k|}{2}\right)^{2}}
=  C_{x,t} \phi_{ \frac{t}{4}}(y),
\end{eqnarray}
where $C_n=2^n\left(1+\sqrt{\frac{\pi}{t}}\right)^{n}$ and $C_{x,t}= C_n(4\pi t)^{\frac{n}{2}}.$ 
On the other hand, if \(|x-y| > |x|\), then \(|y| \leq |x-y|+|x| < 2|x-y|\). Thus,
\begin{eqnarray}\label{ds3}
\phi_{t}(x-y)=\sum_{k \in \mathbb{Z}^n}e^{-\frac{1}{4t}|x-y+k|^{2}} \leq \phi_{ \frac{t}{4}}(y).
\end{eqnarray}
Choosing  \(t_1 = \frac{t_0}{4}\) and combining \eqref{ds0},  \eqref{ds1}, and \eqref{ds3} and using the hypothesis,  we  obtain
\[
\begin{aligned}
&\left\|\phi_{t}(x-\cdot) v^{-\frac{1}{p}}(\cdot)\right\|_{L^{p^{\prime}}(\mathbb T^n)}\\
&\quad\quad\quad\quad\quad\leq C_{x,t}\left\|\chi_{B_{x}} \phi_{t_1}v^{-\frac{1}{p}}\right\|_{L^{p^{\prime}}(\mathbb T^n)}+\left\|\chi_{B_{x}^{c}} \phi_{t_1} v^{-\frac{1}{p}}\right\|_{L^{p^{\prime}}(\mathbb T^n)}<\infty.
\end{aligned}
\]
This proves \eqref{1} implies \eqref{2}.    Part \eqref{2} implies \eqref{3} is obvious.   Suppose \eqref{3} holds true.
 Hence, for any positive function \(f \in L_v^{p}(\mathbb T^n)\), \(\phi_{t_{0}} * f(x) < \infty\) for almost every \(x\). Let \(x_{0} \in \mathbb{T}^{n}\) be such that \(\phi_{t_{0}} * f(x_{0}) < \infty\). 
 We will first show that \(\phi_{\frac{t_{0}}{4}} * f(x) < \infty\) for all \(x \in \mathbb{T}^{n}\).
Indeed, assume first that \(x \neq x_{0}\). 
Note that 
\begin{equation}\label{de1}
\begin{cases} |x-y| \leq |x-x_{0}| \implies\left|y-x_{0}\right| \leq 2\left|x-x_{0}\right|,\\
 |x-y| \geq |x-x_{0}| \implies  \left|x_{0}-y\right| \leq 2|x-y|.
\end{cases}
\end{equation}
By \eqref{de1} and Theorem \ref{crb},  we obtain
\begin{eqnarray*}\phi_{\frac{t_{0}}{4}}(x-y) \leq
\begin{cases}  C_n \frac{\phi_{t_{0}}\left(x_{0}-y\right)}{\phi_{t_{0}}\left(2\left(x-x_{0}\right)\right)} = C_{x} \phi_{t_{0}}\left(x_{0}-y\right) & if \  |x-y| \leq |x-x_{0}|, \\
 \phi_{\frac{t_{0}}{4}}\left(\frac{x_{0}-y}{2}\right) = 4^{n / 2} \phi_{t_{0}}\left(x_{0}-y\right) & if \ |x-y| \geq |x-x_{0}|,
 \end{cases}
\end{eqnarray*}
where $C_x=\frac{C_n}{\phi_{t_{0}}\left(2\left(x-x_{0}\right)\right)} .$ Thus,  we have 
\begin{eqnarray}\label{eq2}
\phi_{\frac{t_0}{4}} \ast f(x)
& \leq & \left(C_{x} + 4^{n / 2}\right) \left(\int_{|x-y|<|x-x_{0}|} + \int_{|x-y| \geq |x-x_{0}|}\right) \phi_{t_{0}}\left(x_{0}-y\right) f(y) d y\nonumber\\
& \lesssim & \phi_{t_0}\ast f (x_0)
 <  \infty  \quad for \ all\ \   x \in \mathbb{T}^{n} \backslash \{x_{0}\}.
\end{eqnarray}
Note that, later inequality in \eqref{de1} also hold for $x=x_0.$ 
That is
\[
0 \leq \int_{\mathbb{T}^{n}} \phi_{\frac{t_{0}}{4}}\left(x_{0}-y\right) f(y) d y \leq 4^{n / 2} \int_{\mathbb{T}^{n}} \phi_{t_{0}}\left(x_{0}-y\right) f(y) d y < \infty.
\]
Thus,   we have
\[\int_{\mathbb{T}^{n}} \phi_{\frac{t_{0}}{4}}(x-y) f(y) d y< \infty \quad for\ all\  \ x\in \mathbb T^n. \] In particular,

\[\int_{Q_n} \phi_{\frac{t_{0}}{4}}(y) f(y) d y< \infty \quad for \  all \   f\in L^p_v(\mathbb T^n).\] Therefore, by duality,  for $t_1=t_0/4,$ we 
\[ \left\|\phi_{t_1} v^{-\frac{1}{p}}\right\|_{L^{p^{\prime}}(\mathbb T^n)}= \sup_{\|f\|_{L^p_v(\mathbb T^n)} \leq 1} \left| \langle \phi_{t_1}v^{-1/p},   fv^{1/p} \rangle \right| < \infty. \]
This proves \eqref{1}.
\end{proof}

\begin{proposition}[characterization of   $D^{WG}_p$]\label{prop-r1}
	Let $v$ be a pisitive weight in $\mathbb{T}^n \times\mathbb{R}^m$and $ 1 \leq p<\infty.$
	Then the  following statements are equivalent:
	\begin{enumerate}
		\item\label{r5} $v\in D^{WG}_p$ for some $ 1 \leq p<\infty.$
		\item\label{r6} There exists $t_0>0$ and a weight $u$ such that the operator $f \rightarrow \phi^{WG}_{t_0} * f$ maps $L_v^p(\mathbb{T}^n \times\mathbb{R}^m)$ into $L_u^p(\mathbb{T}^n \times\mathbb{R}^m)$ for $p>1,$
		and maps $L_v^1(\mathbb{T}^n \times\mathbb{R}^m)$ into weak $L_u^1(\mathbb{T}^n \times\mathbb{R}^m)$ when $p=1$.
		\item\label{r7} There exists $t_0>0$ such that $\phi^{WG}_{t_0} * f(x)$ is finite almost everywhere for all $f \in L_v^p(\mathbb{T}^n \times\mathbb{R}^m)$.
	\end{enumerate}
\end{proposition}
\begin{proof} Since the proof closely follows the proof of the Proposition \ref{prop1}, we provide only a brief outline here.
We shall prove \eqref{r5} implies \eqref{r6}. Similarly in \eqref{eq1}, using H{\"o}lder's inequality, we obtain
	\begin{eqnarray*}
		\| \phi^{WG}_t \ast f\|^p_{L^p_u\left(\mathbb{T}^n \times\mathbb{R}^m\right)} 
		\leq \|  f\|^p_{L_v^p\left(\mathbb{T}^n \times\mathbb{R}^m\right)}  \int_{\mathbb{T}^n \times\mathbb{R}^m} h^t(x,y) u(x,y)\, dx\,dy,
	\end{eqnarray*}
	where $ h^t(x,y)= \left\|\phi^{WG}_{t}((x,y)-(\cdot,\cdot)) v^{-\frac{1}{p}}(\cdot,\cdot)\right\|_{L^{p^{\prime}}(\mathbb{T}^n \times\mathbb{R}^m)}^{p}.$ We claim that  there exists $t_0>0$ such that $h^{t_0}(x,y)$ is finite for all $(x,y)$.  By choosing \(u \in L_{\text{loc}}^{1}(\mathbb{T}^n \times \mathbb{R}^m)\) as
$$
u(x,y)=\left\{\begin{array}{l}
	\quad1 \quad\quad\quad if \ h^{t_0}(x, y) \leq e^{-|(x, y)|^2}, \\
	\frac{e^{-|(x,y)|^2}}{h^{t_0}(x,y)}\quad if\  h^{t_0}(x, y)>e^{-|(x, y)|^2} ,
\end{array}\right.
$$	
we can get \(h^{t_0} \cdot u\) is in \(L^{1}(\mathbb{T}^n \times \mathbb{R}^m)\) and conclude \eqref{r6}.  We now prove our claim.
	For $z:=(x,y)\in \mathbb{T}^n \times\mathbb{R}^m,$ consider
	\[B_{z}:=\left \{z':=(x',y') \in \mathbb{T}^n \times\mathbb{R}^m:|z-z'| \leq|z| \right \}.\]
In \(\mathbb{T}^n \times \mathbb{R}^m\), the modulus is equivalent to the modulus in \(\mathbb{R}^{n+m}\). Consequently, for \(t>0\), we have
	\begin{eqnarray}\label{eq-r23}
		&\left\|\phi^{WG}_{t}(z-\cdot) v^{-\frac{1}{p}}(\cdot)\right\|_{L^{p^{\prime}}(\mathbb{T}^n \times\mathbb{R}^m)} \\
		&\leq\left\|\chi_{B_{z}}(\cdot) \phi^{WG}_{t}(z-\cdot) v^{-\frac{1}{p}}(\cdot)\right\|_{L^{p^{\prime}}(\mathbb{T}^n \times\mathbb{R}^m)}+\left\|\chi_{B_{z}^{c}}(\cdot) \phi^{WG}_{t}(z-\cdot) v^{-\frac{1}{p}}(\cdot)\right\|_{L^{p^{\prime}}(\mathbb{T}^n \times\mathbb{R}^m)}.\nonumber
	\end{eqnarray} 
	If \(|z-z'| \leq |z|\), then \(|z'| \leq |z-z'|+|z| \leq 2|z|\). Following \eqref{ds1} and using (\ref{eq-r17}), we can have
	\begin{equation}\label{eq-r24}
		\phi^{WG}_{t}(z-z') 
		\lesssim e^{\frac{1}{4t}|z|^{2}} \sum_{k \in \mathbb{Z}^n}e^{-\frac{1}{4t}\left\{\left(\frac{|x'+k|}{2}\right)^{2}+\left(\frac{|y'|}{2}\right)^{2}\right\}}
		\lesssim \phi^{WG}_{ \frac{t}{4}}(z').
	\end{equation}
	On the other hand, if \(|z-z'| > |z|\), then \(|z'| \leq |z-z'|+|z| < 2|z-z'|\). Thus,
	\begin{eqnarray}\label{eq-r25}
		\phi^{WG}_{t}(z-z')\leq \phi^{WG}_{ \frac{t}{4}}(z').
	\end{eqnarray}
	Choosing  \(t_1 = \frac{t_0}{4}\) and combining \eqref{eq-r23},  \eqref{eq-r24}, and \eqref{eq-r25} and using the hypothesis,  we  obtain
	\[
	\begin{aligned}
		&\left\|\phi^{WG}_{t}(z-\cdot) v^{-\frac{1}{p}}(\cdot)\right\|_{L^{p^{\prime}}(\mathbb{T}^n \times\mathbb{R}^m)}\\
		&\quad\quad\quad\quad\quad\leq C_{z,t}\left\|\chi_{B_{z}} \phi^{WG}_{t_1}v^{-\frac{1}{p}}\right\|_{L^{p^{\prime}}(\mathbb{T}^n \times\mathbb{R}^m)}+\left\|\chi_{B_{z}^{c}} \phi^{WG}_{t_1} v^{-\frac{1}{p}}\right\|_{L^{p^{\prime}}(\mathbb{T}^n \times\mathbb{R}^m)}<\infty.
	\end{aligned}
	\]
	This proves \eqref{r5} implies \eqref{r6}. Part \eqref{r6} implies \eqref{r7} is obvious.   Suppose \eqref{r7} holds true. Let \(z_{0} \in \mathbb{T}^n \times\mathbb{R}^m\) be such that \(\phi^{WG}_{t_{0}} * f(z_{0}) < \infty\). 
	Let \(z \neq z_{0}\), then
	\begin{equation}\label{eq-r26}
		\begin{cases} |z-z'| \leq |z-z_{0}| \implies\left|z'-z_{0}\right| \leq 2\left|z-z_{0}\right|,\\
			|z-z'| \geq |z-z_{0}| \implies  \left|z_{0}-z'\right| \leq 2|z-z'|.
		\end{cases}
	\end{equation}
	By \eqref{eq-r26} and \eqref{eq-r17},  we obtain
	\begin{eqnarray*}\phi^{WG}_{\frac{t_{0}}{4}}(z-z') \lesssim
		\begin{cases}   \frac{\phi^{WG}_{t_{0}}\left(z_{0}-z'\right)}{\phi^{WG}_{t_{0}}\left(2\left(z-z_{0}\right)\right)} \lesssim \phi^{WG}_{t_{0}}\left(z_{0}-z'\right) & if \  |z-z'| \leq |z-z_{0}|, \\
			\phi^{WG}_{\frac{t_{0}}{4}}\left(\frac{z_{0}-z'}{2}\right) \lesssim \phi^{WG}_{t_{0}}\left(z_{0}-z'\right) & if \ |z-z'| \geq |z-z_{0}|.
		\end{cases}
	\end{eqnarray*}
	Thus,  following \eqref{eq2}, we have 
	\begin{eqnarray*}
		\phi^{WG}_{\frac{t_0}{4}} \ast f(z) 
		 \lesssim\phi^{WG}_{t_0}\ast f (z_0)
		<   \infty  \quad for\ all \ \   z \in ( \mathbb{T}^n \times\mathbb{R}^m)\setminus \{z_{0}\}.
	\end{eqnarray*}
	Since the inequality in \eqref{eq-r26} also hold for $z=z_0,$  we have 
	\[\int_{\mathbb{T}^n \times\mathbb{R}^m} \phi^{WG}_{\frac{t_{0}}{4}}(z-z') f(z') d z'< \infty \quad for\ all \  \ z\in \mathbb{T}^n \times\mathbb{R}^m. \] In particular,
	\[\int_{\mathbb{T}^n \times\mathbb{R}^m} \phi^{WG}_{\frac{t_{0}}{4}}(z') f(z') d z'< \infty \quad for \  all \   f\in L^p_v(\mathbb{T}^n \times\mathbb{R}^m).\] Therefore, by duality,  for $t_1=t_0/4,$ we 
	\[ \left\|\phi^{WG}_{t_1} v^{-\frac{1}{p}}\right\|_{L^{p^{\prime}}(\mathbb{T}^n \times\mathbb{R}^m)}= \sup_{\|f\|_{L^p_v(\mathbb{T}^n \times\mathbb{R}^m)} \leq 1} \left| \langle \phi^{WG}_{t_1}v^{-1/p},   fv^{1/p} \rangle \right| < \infty. \]
	This proves \eqref{r5}.
\end{proof}

\subsection{Proof of Theorem \ref{th1}}
It should be noted that Theorem \ref{th1} is a particular case of  following more generic Theorem \ref{th2}.
\begin{theorem}\label{th2}
Let $v$ be a positive weight in $\mathbb{T}^n,$  $\phi_t$ be as in \eqref{ehk} and $1 \leq p<\infty.$ Define the local maximal operator by
\[
H^*_Rf(x) = \sup_{0<t<R}\left|\phi_t\ast f(x)\right|
\]
for some $0<R<\frac{1}{2}$. Then the  following statements are equivalent:
\begin{enumerate}
\item \label{t1} The weight $v\in D_p^{T}$ for some $ 1 \leq p<\infty.$
\item\label{t2} There exists \(0 < R < \frac{1}{2}\) and a weight \(u\) such that the operator \(H^*_Rf\) maps \(L_v^p(\mathbb T^n)\) into \(L_u^p(\mathbb T^n)\) for \(p > 1,\) and maps \(L_v^1(\mathbb T^n)\) into weak \(L_u^1(\mathbb T^n)\) when \(p=1\).
\item\label{t3} There exists \(0 < R < \frac{1}{2}\) such that the convolution \(\phi_{R} * f(x)\) is finite almost everywhere, and the limit \(\lim_{t \rightarrow 0^+} \phi_{t} * f(x) = f(x)\) exists almost everywhere for all \(f \in L_v^p(\mathbb{T}^n)\).
\item\label{t4} There exists \(0 < R < \frac{1}{2}\) such that \(H^*_Rf(x)\) is finite almost everywhere for all \(f \in L_v^p(\mathbb T^n)\).
\end{enumerate}
\end{theorem}

\begin{proof}
Assume that \eqref{t1} holds. Let \( 0<t < R\). Decompose  \(\phi_t\) into two parts as follows
 \[\phi_t = \phi_t^1 + \phi_t^2,  \quad where \  \phi_t^1 = \phi_t \chi_{\{|x| \leq(2nR)^{1/2}\}}.\]
If \(j_0 \in \mathbb{Z}\) is such that \(2^{j_0}t < R < 2^{j_0+1}t\), then we have
\begin{equation}\label{eq-r4}
\phi_t^1(x)\leq \phi_t(x)\left(\chi_{\{|x| \leq r_{0}\}}(x)+\sum_{j=0}^{j_0} \chi_{\left\{(r_j \leq |x| \leq r_{j+1}\right\}}(x)\right),
\end{equation}
where we denote \(r_j = (2n2^{j}t)^{\frac{1}{2}}\) for $0\leq j\leq j_0$. Using the right-hand side inequality of (\ref{eq8}), we can bound \(\phi_t(x)\) as follows
\[
\phi_t(x) \leq 2^n \left(1+\sqrt{\frac{\pi}{t}}\right)^n e^{-\frac{|x|^2}{4t}} \leq \frac{C_n}{t^{\frac{n}{2}}} e^{-\frac{|x|^2}{4t}},
\]
where \( C_n = 2^n \left(\sqrt{R} + \sqrt{\pi}\right)^n \). Hence, (\ref{eq-r4}) can be rewritten as
\[\phi_t^1(x)\leq\frac{C_n}{t^{\frac{n}{2}}}e^{-\frac{|x|^2}{4t}}\left(\chi_{\{|x| \leq r_{0}\}}(x)+\sum_{j=0}^{j_0} \chi_{\left\{r_j \leq |x| \leq r_{j+1}\right\}}(x)\right).\]
Using the fact that \( e^{-\frac{|x|^2}{4t}} \) is decreasing in the variable \( x \), we can rearrange it as
\[\phi_t^1(x)\leq C_n\left((2n)^{\frac{n}{2}}\frac{1}{(r_{0})^n} \chi_{\{B(0,r_{0})\}}(x)+\sum_{j=0}^{j_0} C_{n,j}\frac{1}{(r_{j+1})^{n}} \chi_{\left\{B(0,r_{j+1})\right\}}(x)\right),\]
where we denote $C_{n,j}=\left(2n2^{j+1}\right)^{\frac{n}{2}} e^{-\frac{n}{2} 2^{j}}$ for $0\leq j\leq j_0.$ By separately considering the positive and negative parts of \( f \in \mathbb{T}^n \), we can, without loss of generality, assume that \( f \geq 0 \). Let \( f^* \) denote its 1-periodic extension. Initially, by convolving \( f ^*\) on both sides and subsequently taking the supremum with respect to \( t \) on both sides, we then utilize definition (\ref{eq-r5}) to obtain
\begin{equation}\label{eq3}
\sup _{t<R} \phi_{t}^{1} * f^*(x)\leq C_n\left((2n)^{\frac{n}{2}} \mathcal{M}_{(2 n R)^{1 / 2}} f^*(x)+\sum_{j=0}^{j_0}C_{n,j}\,\mathcal{M}_{(2 n R)^{1 / 2}} f^*(x)\right).
\end{equation}
Hence, we get
\[
\sup _{t<R} \phi_{t}^{1} * f^*(x) \leq C'_{n} \mathcal{M}_{(2 n R)^{1 / 2}} f^*(x),
\]
where 
\[ C'_{n} = C_n\left((2n)^{\frac{n}{2}}+\sum_{j=0}^{j_0}C_{n,j}\right) < \infty.\]
Therefore, by selecting \( R < \frac{1}{8n} \), we obtain
\begin{equation}\label{eq-r2}
	\sup _{t<R} \phi_{t}^{1} * f^*(x) \leq C'_{n} \mathcal{M}_{\frac{1}{2}} f^*(x).
\end{equation}
For the case of $\phi^2_t$, if we will show that  \(\phi_{t}^{2}(x)\) is increasing in the time interval \((0, R')\)  for some $R'>0$ then, we have
\begin{equation}\label{eq12}
\sup _{0<t<R'} \phi_{t}^{2} * f(x) = \phi_{R'}^{2} * f(x) \leq \phi_{R'} * f(x).
\end{equation}
Hence, by selecting \( R = \min\{\frac{1}{8n}, R'\} \), from equations (\ref{eq-r2}) and (\ref{eq12}), we get
\begin{equation}\label{eq13}
H_{R}^{*} f(x) \lesssim\left(\mathcal{M}^{T} f(x)+\phi_{R} * f(x)\right).
\end{equation}
Let \( u_1 \) be a weight such that Proposition \ref{prop1},\eqref{2} holds. Then the same will also hold for any weight \( u' \) such that \( u'(x) \leq u_1(x) \). Similarly, the boundedness of the maximal operator will hold if the weight is \( u''(x) \leq u_2(x) \) whenever \( u_2 \) satisfies Proposition \ref{GP},\eqref{prop2}. Hence, we can choose \( u \) as \( u(x) = \min \{u_1(x), u_2(x)\} \) so that the condition \eqref{t2} is satisfied.

Now, we aim to show that \( \phi_{t}^{2}(x) \) is increasing in the time interval $(0, R')$ for some \( R' > 0 \). First, we will prove that each $k$-th term \( \frac{1}{t^{\frac{n}{2}}}e^{-\frac{|x+k|^2}{4t}} \) of \( \phi_{t}^{2}(x) \) increases for some \( R'_k \) depending on \( k\). Then, by setting \( R' = \inf_k \{R'_k\} >0\), \( \phi_{t}^{2}(x) \) will be increasing. Consider
\[
\frac{d}{dt} \left( \frac{1}{t^{\frac{n}{2}}}e^{-\frac{|x+k|^2}{4t}} \right) = \frac{1}{t^{\frac{n}{2}+2}}e^{-\frac{|x+k|^2}{4t}} \left( \frac{|x+k|^2 - 2nt}{2t} \right).\]
Thus, we need \( \frac{|x+k|^2 - 2nt}{2t} > 0 \), which simplifies to \( t < \frac{|x+k|^2}{2n} \). Therefore, we can choose $R'=R<\frac{1}{2},$ since $\phi_t^2$ is supported on $\chi_{\{|x| >(2nR)^{1/2}\}}.$   This proves that \eqref{t1} implies \eqref{t2}.

Next, we show that \eqref{t2} implies \eqref{t3}. First, we will proving \(\phi_{t}\ast f \rightarrow f\) uniformly for compactly supported continuous functions \(f\). We can select \(\delta>0\) such that \(|y|<\delta\) implies \(|f(x-y)-f(x)|<\epsilon\).  Since,
\[
\begin{aligned}
\left\|\phi_{t}\right\|_{L^{1}(\mathbb T^n)} =\sum_{k \in \mathbb{Z}^n} \int_{\mathbb T^{n}}(4\pi t)^{-\frac{n}{2}} e^{-\frac{|x+k|^{2}}{4 t}} dx
=\int_{\mathbb{R}^{n}}(4\pi t)^{-\frac{n}{2}} e^{-\frac{|x|^{2}}{4t}} dx =1,
\end{aligned}
\]
we can express
\[
\begin{aligned}
	\left|f * \phi_{t}-f\right| &\leq \int_{|y|<\delta} \phi_{t}(y)|f(x-y)-f(y)| d y+\int_{|y|\geq\delta} \phi_{t}(y)|f(x-y)-f(y)| d y \\
	& \leq \varepsilon+2 B \int_{\delta \leqslant|y|}\left|\phi_{t}(y)\right| d y,
\end{aligned}
\]
where \(B\) denotes the bound of \(f\). If we prove that for every \(\delta>0\),
\begin{equation}\label{eq-r6}
\int_{\delta \leq|x|}\phi_t(x) d x \rightarrow 0, \ as \ t \rightarrow 0^+,
\end{equation}
then we can get our claim. Since we are interested in pointwise convergence at \( t = 0 \), we can restrict \( t < R<\frac{1}{2} .\) Now, using the upper bound inequality from (\ref{eq8}), we have
\begin{equation}\label{eq-r7}
	\int_{\delta\leq|x|} \phi_{t}(x) d x \leq(4\pi)^{\frac{n}{2}}\left(\sqrt{R}+\sqrt{\pi}\right)^n \int_{\delta\leq|x|}(4\pi t)^{-\frac{n}{2}} e^{\frac{-|x|^2}{4 t}}dx.
\end{equation}
Since \( \phi_t'(x) =(4\pi t)^{-\frac{n}{2}} e^{-\frac{|x|^2}{4 t}} \) forms a good kernel, it satisfies
\begin{equation}\label{eq-r8}
\int_{\delta\leq|x|} \phi_{t}^{\prime}(x) d x \rightarrow 0, \ as \ t \rightarrow 0^+.
\end{equation}
Therefore, from (\ref{eq-r7}) and (\ref{eq-r8}), we conclude (\ref{eq-r6}).
We have shown that when $f$ is a compactly supported continuous function, $\phi_{t} * f(x) \rightarrow f(x)$ uniformly as $t \rightarrow 0^+$. Now, let \(f \in L_v^p(\mathbb{T}^n),\) where \(1 \leq p < \infty\). Utilizing the fact that \(\left\|\phi_t * f - f\right\|_{L^p_v(\mathbb T^n)} \rightarrow 0\) as \(t \rightarrow 0^+\), we can get a sequence such that \(f_{t_k} \rightarrow f\) pointwise almost everywhere. Hence, it suffices to show that \(\lim_{t \rightarrow 0^+} \phi_t * f(x)\) exists almost everywhere.
Let's denote
$$\Phi f(x) = \left| \limsup_{t \rightarrow 0^+} \phi_t * f(x) - \liminf_{t \rightarrow 0^+} \phi_t * f(x) \right|.$$
If \(f\) is continuous with compact support, then \(\phi_t * f \rightarrow f\) uniformly, and thus \(\Phi f\) is zero.
Next, if \(f\) is in \(L^p_v(\mathbb{T}^n)\), then by $L^p_v(\mathbb T^n)$ to  weak $L^p_u(\mathbb T^n)$  boundedness of $H^*_R$, we have
$$
m\{x : 2 H^*_Rf > \lambda\} \leq \frac{2C^p}{\lambda} \|f\|_{L^p_v(\mathbb T^n)}.
$$
However, we have
$$
\Phi f(x) \leq 2 H^*_Rf(x),
$$
Hence,
$$
m\{x : \Phi f(x) > \lambda\} \leq \frac{2C^p}{\lambda} \|f\|_{L^p_v(\mathbb T^n)}.
$$
Since continuous compactly supported functions are dense in \(L_v^p(\mathbb{T}^n)\), we can express \(f = f_1 + f_2\) with \(f_1\) being continuously compactly supported and \(L_v^p\) norm of \(f_2\) is arbitrarily small. However, \(\Phi f \leq \Phi f_1 + \Phi f_2\), and \(\Phi f_1 = 0\) since \(f_1\) is continuous. Therefore,
$$
m\{x : \Phi f(x) > \lambda\} \leq \frac{2C^p}{\lambda} \|f_2\|_{L^p_v(\mathbb T^n)}.
$$
Since the norm of \(f_2\) can be arbitrarily small, we conclude \(\Phi f = 0\) almost everywhere, implying that \(\lim_{t \rightarrow 0^+} \phi_t * f(x)\) exists almost everywhere.

Let us assume \((3)\). Let \(x\) be such that \(\lim_{t \rightarrow 0^+} \phi_{t} * f(x)\) exists and \(\phi_{R} * f(x) < \infty\). Then there exist a constant \(0 < C_{x, f} < \infty\) and \(t_{x, f}> 0\) such that
\begin{equation}\label{eq14}
\sup_{t < t_{x, f}} \phi_{t} * f(x) < C_{x, f}.
\end{equation}
We can clearly choose \(t_{x, f} < R\). Now, consider \(t_{x, f} < t < R\). Then
\[
\begin{aligned}
\phi_{t}(x-y) &= (4\pi t)^{-\frac{n}{2}} \sum_{k \in \mathbb{Z}^n} e^{-\frac{|x-y+k|^2}{4t}} \leq (4\pi t_{x,f})^{-\frac{n}{2}}\sum_{k \in \mathbb{Z}^n} e^{-\frac{|x-y+k|^2}{4R}\cdot \frac{R}{t}}\\
&\leq (4\pi t_{x,f})^{-\frac{n}{2}}\sum_{k \in \mathbb{Z}^n} e^{-\frac{|x-y+k|^2}{4R}}\leq \left(\frac{R}{t_{x, f}}\right)^{\frac{n}{2}} \phi_{R}(x-y).
\end{aligned}
\]
Therefore,
\begin{equation}\label{eq15}
\sup_{t_{x, f} < t < R} \phi_{t} * f(x) \leq \left(\frac{R}{t_{x, f}}\right)^{\frac{n}{2}} \phi_{R} * f(x) < \infty.
\end{equation}
Equations (\ref{eq14}) and (\ref{eq15}) give \((4)\). Lastly, Proposition \ref{prop1}, together with the similar arguments in its proof \eqref{3}$\implies$\eqref{1}, establishes the implication from \eqref{t4} to \eqref{t1}.
\end{proof}
\subsection{Proof of Theorem \ref{th-r1}}  In this section we shall prove following more general theorem.
\begin{theorem}\label{th-r2}
Let $v$ be a positive weight in $\mathbb{T}^n \times\mathbb{R}^m,$  $\phi^{WG}_t$ be defined as in  \eqref{eq5} and $1 \leq p<\infty.$ Define the local maximal operator by
\[
W^*_Rf(x,y) = \sup_{0<t<R}\left|\phi^{WG}_t\ast f(x,y)\right|
\]
for some $0<R<\frac{1}{2}$. Then the  following statements are equivalent:
\begin{enumerate}
\item\label{r1} $v\in D^{WG}_p$ for some $ 1 \leq p<\infty.$
\item\label{r2} There exists \(0 < R < \frac{1}{2}\) and a weight \(u\) such that the operator \(W^*_Rf\) maps \(L_v^p(\mathbb{T}^n \times\mathbb{R}^m)\) into \(L_u^p(\mathbb{T}^n \times\mathbb{R}^m)\) for \(p > 1,\) and maps \(L_v^1(\mathbb{T}^n \times\mathbb{R}^m)\) into weak \(L_u^1(\mathbb{T}^n \times\mathbb{R}^m)\) when $p=1$.
\item\label{r3} There exists \(0 < R < \frac{1}{2}\) such that the convolution \(\phi^{WG}_{R} * f\) is finite almost everywhere, and the limit \(\lim _{t \rightarrow 0^+} \phi^{WG}_{t} * f(x,y) = f(x,y)\) exists almost everywhere for all \(f \in L_v^p(\mathbb{T}^n \times\mathbb{R}^m)\).
\item\label{r4} There exists \(0 < R < \frac{1}{2}\) such that \(W^*_Rf(x,y)\) is finite almost everywhere for all \(f \in L_v^p(\mathbb{T}^n \times\mathbb{R}^m)\).
\end{enumerate}
\end{theorem}
\begin{remark}\label{fr1}
In Theorems \ref{th2} and  \ref{th-r2}, we are not only characterizing the weighted \(L^p\)-spaces on   \(\mathbb{T}^n\) and  \(\mathbb{T}^n \times \mathbb{R}^m\) respectively,  for which the solution of the heat equations converges pointwise to the initial data,  in fact,  we also show that  the  same weighted \(L^p\)-spaces characterize the boundedness of the corresponding maximal operators.
\end{remark}
\begin{remark}\label{fr2} We note the following weight class inclusions for maximal functions
\begin{enumerate}
\item We have $D^T_p= \{ v: v^{-\frac{1}{p}} \in L^{p'}(\mathbb T^n)\}$ for $1\leq p<\infty,$ where $D^T_p$ is weight class for boundedness of $\mathcal{M}^{T}$ (see Theorem \ref{GP},\eqref{prop2}) and $ \{ v: v^{-\frac{1}{p}} \in L^{p'}(\mathbb T^n)\}$ is same for $H^*_R$ (see Theorem \ref{th2}).
\item For waveguide manifold, we have  $D^{WG}_p\subsetneq \{ v: v^{-\frac{1}{p}} \in L^{p'}_{loc}(\mathbb{T}^n \times \mathbb{R}^m)\}.$ This follows from the inequality
\[\chi_K(x)v^{-\frac{1}{p}}(x) \lesssim  \chi_K(x) e^{-\frac{|x|^2}{4t_0}} v^{-\frac{1}{p}}(x) \lesssim  e^{-\frac{|x|^2}{4t_0}} v^{-\frac{1}{p}}(x),\] 
where $K$ is compact and $t_0>0.$
On the other side, \(v(x, y) = e^{-\frac{y^2}{4t_0}p}\) belongs to \(L_{loc}^{p'}(\mathbb{T}^n \times \mathbb{R}^m)\) but not to \(D^{WG}_p\).
\end{enumerate}
\end{remark}

\begin{proof}[{\bf Proof of Theorem \ref{th-r2}}]
Assuming \eqref{r1}, we can choose without loss of generality that \(f \geq 0\). Let \(0 < t < R\). Decompose \(\phi^{WG}_t\) as 
\[\phi^{WG}_t = \phi^{WG}_{t,1} + \phi^{WG}_{t,2},  \quad where \  \phi^{WG}_{t,1}= \phi^{WG}_t \chi_{\{|(x,y)| \leq(2R(n+m))^{\frac{1}{2}}\}}.\]
If \(j_0 \in \mathbb{Z}\) is such that \(2^{j_0}t < R < 2^{j_0+1}t\), then we have
\begin{equation}\label{eq-r18}
\phi^{WG}_{t,1}(x,y)\leq \phi^{WG}_t(x,y)\left(\chi_{\{|(x,y)| \leq r_{0}\}}(x,y)+\sum_{j=0}^{j_0} \chi_{\left\{(r_j \leq |(x,y)| \leq r_{j+1}\right\}}(x,y)\right),
\end{equation}
where we denote \(r_j = (2(n+m)2^{j}t)^{\frac{1}{2}}\) for $0\leq j\leq j_0$. Using the right-hand side inequality of (\ref{eq-r17}), we can bound \(\phi^{WG}_t(x,y)\) as follows
\[
\phi^{WG}_t(x,y) \leq 2^n\left(1+\sqrt{\frac{\pi}{t}}\right)^{n}\left(\frac{1}{4\pi t}\right)^{\frac{m}{2}}e^{-\frac{|(x,y)|^2}{4t}}\leq \frac{C_{n,m}}{t^{\frac{n+m}{2}}} e^{-\frac{|(x,y)|^2}{4t}},
\]
where \( C_{n,m} = 2^n \left(\sqrt{R} + \sqrt{\pi}\right)^n (\sqrt{4\pi})^{-m}\). Hence, (\ref{eq-r18}) can be rewritten as
\[\phi^{WG}_{t,1}(x,y)\leq\frac{C_{n,m}}{t^{\frac{n+m}{2}}}e^{-\frac{|(x,y)|^2}{4t}}\left(\chi_{\{|(x,y)| \leq r_{0}\}}(x,y)+\sum_{j=0}^{j_0} \chi_{\left\{r_j \leq |(x,y)| \leq r_{j+1}\right\}}(x,y)\right).\]
Using the fact that \( e^{-\frac{|(x,y)|^2}{4t}} \) is decreasing for fixed $t$, we can rearrange it as
\[\phi^{WG}_{t,1}(x,y)\leq C_{n,m}\left(C_0\frac{1}{r_{0}^{n+m}} \chi_{\{B(0,r_{0})\}}(x,y)+\sum_{j=0}^{j_0} C_{j+1}\frac{1}{r_{j+1}^{n+m}} \chi_{\left\{B(0,r_{j+1})\right\}}(x,y)\right),\]
where we denote $C_0=(2(n+m))^{\frac{n+m}{2}}$ and $C_{j+1}=\left(2(n+m)2^{j+1}\right)^{\frac{n+m}{2}} e^{-\frac{n+m}{2} 2^{j}}$ for $0\leq j\leq j_0.$ Similarly, in \eqref{eq3}, we get
\[\sup _{t<R} \phi^{WG}_{t,1} * f^*(x,y)\leq C_{n,m}\left(C_0 \mathcal{M}_{R_{n,m}^{1 / 2}} f^*(x,y)+\sum_{j=0}^{j_0}C_{j+1}\,\mathcal{M}_{R_{n,m}^{1 / 2}} f^*(x,y)\right),\]
where $R_{n,m}=2R(n+m).$ Hence, we get
\[
\sup _{t<R}  \phi^{WG}_{t,1} * f^*(x,y) \leq C'_{n,m} \,\mathcal{M}_{R_{n,m}^{1 / 2}} f^*(x,y),
\]
where 
\[ C'_{n,m} = C_{n,m}\left(C_0+\sum_{j=0}^{j_0}C_{j+1}\right) < \infty.\]
Therefore, by choosing a suatable $R'$, we obtain
\begin{equation}\label{eq-r19}
\sup _{t<R'} \phi^{WG}_{t,1} * f^*(x,y) \lesssim \mathcal{M}_{\frac{1}{2}} f^*(x,y).
\end{equation}
For the case of $\phi^{WG}_{t,2}$, we can easily get that  \(\phi^{WG}_{t,2}(x,y)\) is increasing in the time interval \((0, R'')\)  for some $R''>0$ then, we have
\begin{equation}\label{eq-r20}
	\sup _{0<t<R''} \phi^{WG}_{t,2} * f(x,y) = \phi^{WG}_{R'',2}* f(x,y) \leq \phi^{WG}_{R''} * f(x,y).
\end{equation}
Hence, by selecting \( R = \min\{R', R''\} \), from equations (\ref{eq-r19}) and (\ref{eq-r20}), we obtain
\begin{equation}\label{eq13}
	W_{R}^{*} f(x,y) \lesssim\left(\mathcal{M}^{WG} f(x,y)+\phi^{WG}_{R} * f(x,y)\right).
\end{equation}
The condition (\ref{r2}) follows by applying Proposition \ref{prop-r1} and Proposition \ref{GP},\ref{prop-r2}.
\end{proof}
Similar to how we proved that (\ref{t2}) leads to (\ref{t3}) in Theorem \ref{th2}, we use the fact that continuous functions with compact support are dense in \(L^{p}_v(\mathbb{T}^n \times\mathbb{R}^m)\) and \(L^p_v(\mathbb{T}^n \times\mathbb{R}^m)\) to weak \(L^p_u(\mathbb{T}^n \times\mathbb{R}^m)\) boundedness of \(W^*_R\) leads us from (\ref{r2}) to (\ref{r3}).

Let us assume (\ref{r3}). Let \((x,y)\) be such that \(\lim_{t \rightarrow 0^+} \phi^{WG}_{t} * f(x,y)\) exists and \(\phi^{WG}_{R} * f(x,y) < \infty\). Then there exist constants \(0 < C_{x, y} < \infty\) and \(t_{x, y}> 0\) such that
\begin{equation}\label{eq-r21}
	\sup_{t < t_{x, y}} \phi^{WG}_{t} * f(x,y) < C_{x, y}.
\end{equation}
We can clearly choose \(t_{x, y} < R\). Now, consider \(t_{x, y} < t < R\). Then, we can get
\[
\phi^{WG}_{t}((x,y)-(x',y')) \leq \left(\frac{R}{t_{x, y}}\right)^{\frac{n+m}{2}} \phi^{WG}_{R}((x,y)-(x',y')).
\]
Therefore,
\begin{equation}\label{eq-r22}
	\sup_{t_{x, y} < t < R} \phi^{WG}_{t} * f(x,y) \leq \left(\frac{R}{t_{x, y}}\right)^{\frac{n+m}{2}} \phi^{WG}_{R} * f(x,y) < \infty.
\end{equation}
Equations (\ref{eq-r21}) and (\ref{eq-r22}) give (\ref{r4}). Finally, from Proposition \ref{prop-r1} we can get the implication (\ref{r4}) implies (\ref{r1}).

\bigskip

\noindent{\bf Acknowledgements:} The second author is thankful to DST-INSPIRE (DST/ INSPIRE/04/2016/001507) for the financial support provided by D.G.B.. The second author would also like to gratefully acknowledge the support provided by IISER Pune, Government of India.

\bigskip

\bibliographystyle{plain}
\bibliography{version_6.bib}

\end{document}